\documentclass[11pt,letterpaper]{amsart}
\usepackage{amsmath}
\usepackage{amsfonts}
\usepackage{amssymb,amscd,amsthm}
\usepackage{geometry}
\usepackage{mathrsfs}
\usepackage[bookmarksnumbered, colorlinks, linktocpage, plainpages]{hyperref}
\usepackage{nameref}
\usepackage{wasysym,amssymb,indentfirst,color}
\usepackage{enumitem}
\usepackage{mathtools}
\usepackage{appendix}
\usepackage{extarrows}
\usepackage{setspace}

\usepackage{cases}
\usepackage{graphicx}
\usepackage{ragged2e}

\numberwithin{equation}{section}

\newcommand{\algma}{\mathcal{U}}
\newcommand{\Hom}{\text{Hom}}

\newcommand{\bfs}{Without loss of generality we can assume}

\newcommand{\spc}{\mathbb{C}}
\newcommand{\spr}{\mathbb{R}}
\newcommand{\spo}{\mathbb{O}}

\newcommand\huab[2]{\mathscr{B}_{\mathcal{RO}}(#1,#2)}
\newcommand\huabr[2]{\mathscr{B}_{\spr}(#1,#2)}

\newcommand{\re}{\text{Re}\,}%
\newcommand{\im}{\text{Im}\,}
\newcommand\fx[2]{\left<#1,#2\right>}%

\newcommand\fsh[1]{\|#1\|}

\def\S{\mathbb{S}}
\def\O{\mathbb{O}}
\def\R{\mathbb{R}}

\def\abs#1{\left|#1\right|}

\newtheorem{mydef}{Definition}[section]
\newtheorem{rem}[mydef]{Remark}
\newtheorem{eg}[mydef]{Example}
\newtheorem{cor}[mydef]{Corollary}
\newtheorem{prop}[mydef]{Proposition}
\newtheorem{lemma}[mydef]{Lemma}
\newtheorem{thm}[mydef]{Theorem}

\newtheorem{step }[stp]{Step }

\begin{document}
	\title[Octonionic self-adjoint operators]{Octonionic Para-linear Self-Adjoint Operators and Spectral Decomposition}	
	\author{Qinghai Huo}
	
	\email[Q.~Huo]{hqh86@mail.ustc.edu.cn}
	\address{Department of Mathematics, Hefei University of Technology, Hefei 230009, China}
	
	\author{Guangbin Ren}
	\email[G.~Ren]{rengb@ustc.edu.cn}
	\address{Department of Mathematics, University of Science and Technology of China, Hefei 230026, China}

	\author{Irene Sabadini}

	\email[I.~Sabadini]{irene.sabadini@polimi.it}
	\address{Politecnico di Milano, Dipartimento di Matematica, Via E. Bonardi 9, 20133 Milano, Italy}

	\date{}

	\thanks{This work was supported by the National Natural Science Foundation of China (Nos. 12571090, 12301097) and the Fundamental Research Funds for the Central Universities (No. 	JZ2025HGTB0171). The third author is partially supported by PRIN 2022 {\em Real and Complex Manifolds: Geometry and Holomorphic Dynamics} and is member of GNASAGA of INdAM}
	

	\begin{abstract}
	This paper presents a groundbreaking advancement in the theory of operators defined on octonionic Hilbert spaces, successfully resolving a fundamental challenge that has persisted for over six decades. Due to the  {intrinsic} non-associative nature of octonions, conventional linear operator theory encounters profound structural difficulties. We  {make use of an original} conceptual framework termed ``para-linearity", an innovative generalization of linearity that naturally accommodates the octonionic algebraic structure.
	
	Within this newly established paradigm, we systematically develop an appropriate algebraic setting by defining a carefully designed operator algebra and an  adjoint operation which, together, recapture essential analytic properties previously inaccessible in this context.  We identify a geometric structure—the ``slice cone"—as the fundamental object encoding spectral properties typically derived through sesquilinear forms.
	
	We obtain a rigorous characterization of self-adjointness which indicates how to  introduce a new notion of strong eigenvalues. For every compact, para-linear, self-adjoint operator with strong eigenvalues,  we can establish the spectral decomposition theorem and functional calculi.

\end{abstract}

	\maketitle
	
	\tableofcontents

\section{Introduction}

The octonion algebra $\mathbb{O}$, as the largest normed division algebra, occupies a singular position at the {crossroad} of algebra, geometry, and theoretical physics. Its non-associative structure underpins {very important} phenomena across diverse domains: geometric analysis, quantum mechanics, and   mathematical physics   \cite{MR4855317,MR4782804,MR4771382,baez2002octonions,bryant2003some, Duff1999world,Furey2022standardmod}.  Despite this ubiquity, a coherent spectral theory for operators on octonionic Hilbert spaces has remained  {an open problem} since {1964, when} Goldstine and Horwitz's {published their} pioneering work \cite{goldstine1964hilbert}, due to the fundamental incompatibility of associative functional-analytic tools with octonionic non-associativity.

\subsection{\bf Motivations {from Physics}:} Algebraic formulations of quantum theory are fundamentally rooted in $C^*$-algebras and their representations, {see e.g.} \cite{Moretti2019quantum}. Standard quantum mechanics assumes a complex Hilbert space, whose algebraic structure reflects specific physical symmetries. To address questions beyond the Standard Model---such as quantum gravity or true unification---physicists increasingly explore broader normed division algebras. The octonions $\mathbb{O}$ naturally arise:

\begin{quote}
``Since the heart of quantum theory is a $C^*$-algebra and its representations, it is natural to ask what new physics might be hidden in other normed division algebras. The octonions, being the largest and most exceptional, may encode symmetries lying beyond the Standard Model'' \cite{Moretti2019quantum}.
\end{quote}
Physically, non-associativity mirrors the non-locality expected in quantum gravity, while $\mathbb{O}$'s exceptional symmetries ($\text{Aut}(\mathbb{O}) = G_2 \subset \text{SO}(7)$)
may be a natural framework for Grand Unified Theories.  Key manifestations include:
\begin{itemize}
\item \textit{10D superstring backgrounds:} Spacetime factorizes as $\mathbb{R}^{1,3} \times \mathbb{O}$, with octonionic directions describing six compactified dimensions \cite{Duff1999world}.
\item \textit{Exceptional Lie algebras:} $E_8$ is constructed from the Jordan algebra $\mathrm{J}_3(\mathbb{O})$, suggesting $\mathbb{O}$ as the natural arena for maximal unified gauge symmetries \cite{baez2002octonions}.
\item \textit{Standard Model extensions:} Octonionic structures accommodate exactly 72 fermion states---matching the three generations of the Standard Model \cite{Furey2022standardmod}.
\end{itemize}

 \subsection{\bf {Mathematical} Obstacles:} Extending quantum mechanics to octonionic Hilbert spaces faces three core obstructions:
\begin{enumerate}
\item \textit{Spectral theory:} The lack of associativity prevents projection-valued measures, obstructing the spectral theorem.
\item \textit{Probability interpretation:} Gleason's theorem fails as non-associativity undermines additive probability measures on the projection lattice.
\item \textit{Canonical commutation relations:} Operator ordering ambiguities render consistent formulations of the {classical commutation relation $[x,p] = i\hbar$ between the position operator $x$ and the momentum operator $p$}  highly nontrivial.
\end{enumerate}
Subsequent workarounds---Jordan algebras \cite{Baez2012FP}, twisted tensor models \cite{Dixon1994division}, geometric entanglement \cite{gunaydin2013FP}---either restricted to associative subalgebras or embedded $\mathbb{O}$ into larger structures. A unified spectral framework remained open.

 \subsection{\bf Para-Linearity:  A Geometric Resolution.} We resolve these challenges through \emph{para-linearity}, a geometrically motivated non-associative analog of linearity. In the theory of octonionic Hilbert spaces, para-linear rather than octonionic-linear operators are the natural objects, as demonstrated by the octonionic Riesz representation theorem  \cite{huoqinghai2022Riesz}.
  For a map $T : H_1 \to H_2$ between $\mathbb{O}$-modules, this is characterized by:
\begin{equation} \label{eq:paralinear}
\operatorname{Re} \left( T(x)p - T(xp) \right) = 0, \quad \forall x \in H_1,  p \in \mathbb{O}.
\end{equation}
This condition preserves the geometric essence of linearity while respecting octonionic non-associativity.
Here \(\operatorname{Re}\) denotes the real part in an \(\mathbb{O}\)-module \cite{huoqinghai2020nonass}. This condition retains the geometric core of linearity while accommodating octonionic non-associativity. Notably, when the octonion algebra \(\mathbb{O}\) is replaced by the quaternion algebra \(\mathbb{H}\), para-linearity reduces to standard quaternionic linearity \cite{huoqinghai2022Riesz}. Within this framework, we recover several cornerstone results:

\begin{itemize}
\item Riesz representation theorem \cite{huoqinghai2022Riesz},
\item Parseval theorem \cite{huoqinghai2021tensor},
\item Hahn-Banach theorem \cite{huo2025BLMSHB}.
\end{itemize}
The concept of para-linearity also finds applications in octonionic function theory, as demonstrated in works such as \cite{Colombo2023OctonionicMA,MR4935003,Krausshar2022discoct,Krauhar2023WeylCP}.

\subsection{Principal Contributions.} Within this framework, {in this paper} we establish the first complete spectral theory for self-adjoint operators in octonionic Hilbert spaces. Our main results are:
\begin{enumerate}
\item

\textit{Involutive operator algebra.} The space \(\mathscr{B}_{\mathcal{R}\mathcal{O}}(H)\) of bounded para-linear operators on an octonionic bimodule \(H\) is an octonionic involutive Banach algebra with respect to the regular composition \(\circledcirc\) defined in Definition~\ref{def:regular_composition}. The involution is given by the corrected adjoint \(T \mapsto T^{*}\), characterized by
\[
\langle x, T^{*} y \rangle \;=\; \langle T x, y \rangle \;+\; [y,T,x],
\]
for all \(x,y \in H\). These properties are established in Theorems~\ref{thm:OBanach}, \ref{thm:norm_estimate}, and \ref{thm:involutive_algebra}.

\item \textit{Geometric self-adjointness:} A para-linear operator $T$ is self-adjoint if and only if $\langle Tx, x \rangle \in \mathbb{R}$ for all $x$ in the \emph{slice cone} $\mathbb{C}(H) := \bigcup_{J \in \mathbb{S}} (\operatorname{Re} H + J \operatorname{Re} H)$, where $\mathbb{S} = \{J \in \operatorname{Im} \mathbb{O} : |J| = 1\}$  {(Theorem  \ref{thm:slice_cone_char})}.  This is established via a sharp octonionic polarization identity (Theorem \ref{thm:polarization}).

\item \textit{Spectral decomposition:} Every compact self-adjoint para-linear operator with \emph{standard strong eigenvalues} admits a complete spectral decomposition:
\begin{equation} \label{eq:spectral}
T = \sum_{k=1}^{\infty} \lambda_k P_{z_k},
\end{equation}
where $\lambda_k \in \mathbb{R}$, $z_k \in \mathbb{C}(H)$ are eigenvectors, and $P_{z_k}(x) = z_k \langle z_k, x \rangle$ are para-linear projections (Theorem  \ref{thm:hilbert-schmidt}). Eigenvalues are real (Corollary \ref{StrongEigen}), and eigenvectors are orthogonal with $B_p(z_i, z_j) = 0$ (Theorem \ref{thm:eigen-orthogonality}).

\item \textit{Functional calculus:} For real-analytic functions $f$, we define decomposition-independent functional calculi:
\begin{equation*}
\Phi(f) = \sum_k P_{z_k} \odot f(\lambda_k), \quad \Psi(f) = \sum_k f(\lambda_k) \odot P_{z_k},
\end{equation*}
satisfying $\Phi(q^k) = T^{\circledcirc k}$, $\operatorname{Re} \Phi(f  g) = \operatorname{Re} (\Phi(f) \circledcirc \Phi(g))$ for $f \in \operatorname{Re} \mathcal{R}$, and $\|\Phi(f)\| \leq 8 \|f\|_\infty$ (Theorems \ref{thm:functional_calculus_properties}, \ref{thm:independence}).
\end{enumerate}

\subsection{Impact and Applications.}

This work transforms non-associativity from an obstruction into a structured feature:
\begin{itemize}
\item It resolves problems open since Goldstine--Horwitz (1964) and provides foundations for non-associative functional analysis.
\item The slice cone $\mathbb{C}(H)$ offers a universal geometric domain for spectral analysis across non-associative settings.
\item Applications include $G_2$-manifold analysis, {string compactifications}, and non-associative quantum mechanics.
\end{itemize}
In summary, para-linearity establishes the definitive operator theory for octonionic Hilbert spaces, enabling spectral analysis where associative methods fail and opening pathways in  non-associative functional analysis.

\section{Algebraic Foundations of Para-Linear Operators}
\label{sec:algebraic_foundations}

This section establishes the algebraic framework for para-linear operators in octonionic Hilbert spaces. We begin by recalling essential properties of the octonion algebra, then define module structures and para-linearity, and {finally} Hilbert $\mathbb{O}$-bimodules and the Riesz representation theorem.

\subsection{Octonionic algebra}
The octonions $\mathbb{O}$ form an $8$-dimensional normed division algebra over $\mathbb{R}$ with basis $\{e_0=1, e_1, \dots, e_7\}$ {satisfying the following}\\
{\textbf{Fano Plane Multiplication:} Basis elements multiply via cyclic rules encoded in the Fano plane:}
    \[
    e_i e_j = \epsilon_{ijk} e_k - \delta_{ij}, \quad i,j,k \geq 1,
    \]
    where $\epsilon_{ijk} = 1$ for oriented triples $(123)$, $(145)$, $(176)$, $(246)$, $(257)$, $(347)$, $(365)$.
    	\

    \noindent \small{\textbf{Fig.1} Fano mnemonic graph}
    \begin{flushright}
    	\centerline{\includegraphics[width=3cm]{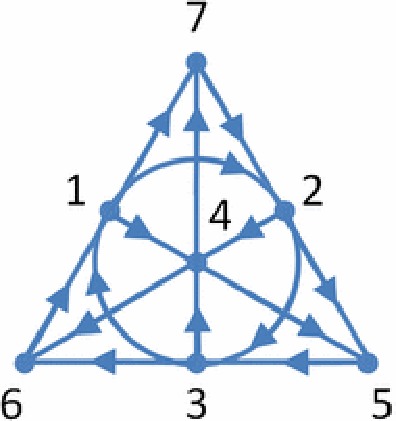}}
    \end{flushright}
    \normalsize

  Key properties include:

\begin{enumerate}
    \item \textit{Conjugation and Norm:} For $x = x_0 + \sum_{i=1}^7 x_i e_i$,
    \[
    \overline{x} := x_0 - \sum_{i=1}^7 x_i e_i, \quad |x| := \sqrt{x\overline{x}}.
    \]

    \item \textit{Non-Associative Structure:} The associator $[x,y,z] := (xy)z - x(yz)$ is fully antisymmetric. The commutator $[x,y] = xy - yx$ measures non-commutativity.

    \item \textit{Moufang Identities:} For all $x,y,z \in \mathbb{O}$:
    \[
    (xyx)z = x(y(xz)), \quad
    z(xyx) = ((zx)y)x, \quad
    x(yz)x = (xy)(zx).
    \]

    \item \textit{Five-Terms Identity:} For all $x,y,z,w \in \mathbb{O}$:
    \begin{equation}
    \label{eq:five_term}
    x[y,z,w] + [x,y,z]w = [xy,z,w] - [x,yz,w] + [x,y,zw].
    \end{equation}

    \
\end{enumerate}

\subsection{Octonionic Module Theory}
We now define {a} module structure compatible with octonionic non-associativity.

\begin{mydef}
An $\mathbb{R}$-vector space $M$ is a \textit{right $\mathbb{O}$-module} if {it is} equipped with a right scalar multiplication satisfying the alternativity  of right associators:
\[
[x,p,q] := (xp)q - x(pq) = -[x,q,p], \quad \forall x \in M, \, p,q \in \mathbb{O}.
\]
Similar definition for {a} \textbf{left $\mathbb{O}$-module}.
\end{mydef}

\begin{mydef}
A right $\mathbb{O}$-module $M$ is an \textit{$\mathbb{O}$-bimodule} if  there is a left   multiplication satisfying
\[
[p,q,x] = [q,x,p] = [x,p,q] = -[q,p,x], \quad \forall p,q \in \mathbb{O}, \, x \in M.
\]
Here the left associators and middle associators are defined canonically.
\end{mydef}

The regular bimodule $\operatorname{Reg}\mathbb{O}$,  {which  as a set is just the $\O$, and scalar multiplications wherein are just defined by the product in $\O$, }is the unique irreducible $\mathbb{O}$-bimodule \cite{Springer2000jordanalg}.  This enables the definition of the real part of an $\mathbb{O}$-bimodule \cite{huoqinghai2020nonass}:
\begin{prop}[Real Part Operator]
\label{prop:real_part}
For an $\mathbb{O}$-bimodule $M$, the \textbf{real part} $\operatorname{Re} M$ is defined as
\[
\operatorname{Re} M := \{m \in M : [p,q,m] = 0,  \ \forall p,q \in \mathbb{O}\}.
\]
The projection $\operatorname{Re}: M \to \operatorname{Re} M$ satisfies:
\begin{enumerate}
    \item $\operatorname{Re} M = \mathscr{Z}(M) := \{m \in M : pm = mp,  \ \forall p \in \mathbb{O}\}$.
    \item $M \cong \operatorname{Re} M \otimes_{\mathbb{R}} \mathbb{O}$ as $\mathbb{O}$-bimodules.
    \item $\operatorname{Re}$ is given explicitly by
    \begin{equation}
    \label{eq:real_part_formula}
    \operatorname{Re} {m = \frac{5}{12}m - \frac{1}{12} \sum_{i=1}^7 e_i m e_i, \ \forall m \in M .}
    \end{equation}
\end{enumerate}
\end{prop}

\subsection{Para-Linear Maps}
Para-linearity generalizes linearity to the non-associative setting by constraining the associator's real part. The notion of para-linearity is first introduced  in \cite{huoqinghai2022Riesz} and then extended into a more general setting in \cite{huoqinghai2020nonass}.

\begin{mydef}[Right Para-Linear Map]
\label{def:para_linear}
Let $M$ be a right $\mathbb{O}$-module and $M'$ an $\mathbb{O}$-bimodule. A map $f \in \operatorname{Hom}_{\mathbb{R}}(M,M')$ is \textbf{right para-linear} if
\[
\operatorname{Re} B_p(f,x) = 0, \quad \forall p \in \mathbb{O}, \, x \in M,
\]
where the \textbf{second right associator} {$B_p$ is defined by}
\[
B_p(f,x) := f(x)p - f(xp),
\]
and $\operatorname{Re}: M' \to \operatorname{Re} M'$ is the real part operator. The set of right para-linear maps is denoted by $\operatorname{Hom}_{\mathcal{RO}}(M,M')$. Left para-linearity is defined analogously.
\end{mydef}
{For next result we refer to Theorem 3.10 in \cite{huoqinghai2020nonass}.}

\begin{thm}[Characterization of para-linearity]
\label{thm:para_linear_char}
Let \(M\) be a right \(\mathbb{O}\)-module and \(M'\) an \(\mathbb{O}\)-bimodule. For \(f \in \operatorname{Hom}_{\mathbb{R}}(M,M')\), let its unique decomposition with respect to the standard basis \(\{1,e_{1},\dots,e_{7}\}\) of \(\mathbb{O}\) be
\[
f(x)=f_{\mathbb{R}}(x)+\sum_{i=1}^{7} f_{i}(x)\,e_{i},
\]
where \(f_{\mathbb{R}}(x), f_{i}(x)\in \operatorname{Re} M'\) for all \(x\in M\) and \(i=1,\dots,7\). Then the following are equivalent:
\begin{enumerate}
    \item \(f \in \operatorname{Hom}_{\mathcal{R}\mathcal{O}}(M,M')\), i.e., \(f\) is right para-linear.
    \item For each \(i=1,\dots,7\) and all \(x\in M\),
    \[
    f_{i}(x) = -\,f_{\mathbb{R}}(x e_{i}).
    \]
    \item For all \(p\in \mathbb{O}\) and \(x\in M\),
    \[
    B_{p}(f,x)=\sum_{i=1}^{7} f_{\mathbb{R}}\!\big([x,p,e_{i}]\big)\,e_{i}.
    \]
\end{enumerate}
\end{thm}

\begin{lemma}[Uniqueness  {Lemma} \cite{huoqinghai2020nonass}] {Let $M$ be a right $\mathbb{O}$-module, $M'$ be an $\mathbb{O}$-bimodule and let}
\label{cor: A_p(x,f)=0 and f(px)=pf(x)}
 $f \in \operatorname{Hom}_{\mathcal{RO}}(M,M')$. Then:
\begin{enumerate}
    \item $f_{\mathbb{R}} = 0$ if and only if $f = 0$.
    \item If $M$ is an $\mathbb{O}$-bimodule, then $f|_{\operatorname{Re} M} = 0$ if and only if $f = 0$.
\end{enumerate}
\end{lemma}
\begin{thm}[$\mathbb{O}$-Bimodule Structure \cite{huoqinghai2020nonass}]
\label{thm:para_bimodule}
Let  $M,M'$ be two $\mathbb{O}$-bimodules. The space $\operatorname{Hom}_{\mathcal{RO}}(M,M')$ is an $\mathbb{O}$-bimodule under {the left and the right multiplication by a scalar defined by:}
\[
(p \odot f)(x) := p f(x) + B_p(f,x), \quad
(f \odot p)(x) := f(px) - B_p(f,x).
\]
\end{thm}

\subsection{Hilbert $\mathbb{O}$-Bimodules}
We now introduce Hilbert structures compatible with para-linearity. For more details see \cite{huoqinghai2022Riesz,huoqinghai2021tensor}.

\begin{mydef}[Pre-Hilbert Right $\mathbb{O}$-Module  \cite{huoqinghai2022Riesz}]
\label{def:pre_hilbert}
A right $\mathbb{O}$-module $H$ is a \textit{pre-Hilbert right $\mathbb{O}$-module} if equipped with an $\mathbb{R}$-bilinear map $\langle \cdot, \cdot \rangle: H \times H \to \mathbb{O}$ (a \textit{right $\mathbb{O}$-inner product}) satisfying:
\begin{enumerate}
    \item \textit{(Para-linearity)} For each $u \in H$, the map $\langle u, \cdot \rangle$ is right $\mathbb{O}$-para-linear.
    \item \textit{($\mathbb{O}$-Hermiticity)} $\langle u, v \rangle = \overline{\langle v, u \rangle}$ for all $u,v \in H$.
    \item \textit{(Positivity)} $\langle u, u \rangle \in \mathbb{R}^+$ and $\langle u, u \rangle = 0$ if and only if $u = 0$.
\end{enumerate}
The norm is $\|u\| = \sqrt{\langle u, u \rangle}$. Completeness with respect to this norm defines a \textit{Hilbert right $\mathbb{O}$-module}. An $\mathbb{O}$-bimodule with a right $\mathbb{O}$-inner product is a \textit{pre-Hilbert $\mathbb{O}$-bimodule}.
\end{mydef}

For \(u,v \in H\) and \(p \in \mathbb{O}\), the \emph{second (right) associator} is defined by
\begin{equation}\label{eq:associator_H}
B_{p}(u,v)\;:=\;B_{p}(\langle u,\cdot\rangle,v)\;=\;\langle u, v\rangle\,p\;-\;\langle u, v p\rangle .
\end{equation}
This notation differs from that in \cite{huoqinghai2022Riesz,huoqinghai2021tensor} in order to avoid confusion with the composition associator (see Definition~\ref{def:associator}). In particular, one has \(B_{p}(u,u)=0\) for all \(u\in H\) and \(p\in\mathbb{O}\) (see~\eqref{eq:Bp(uv)=-Bp(vu)}), which shows that one axiom in the definition of Hilbert spaces given by Goldstine–Horwitz in 1964 is redundant \cite{goldstine1964hilbert}.

{We set} $$\langle x, y \rangle_{\mathbb{R}} := \operatorname{Re} \langle x, y \rangle .$$
{Reasoning as in \cite{huoqinghai2022Riesz} one may prove:}
\begin{lemma}[Inner Product Identities]
\label{lem:inner_product}
{Let $H$ be a pre-Hilbert right $\mathbb{O}$-module.}
For $u,v \in H$, $p,q \in \mathbb{O}$ {we have}:
\begin{enumerate}
    \item $B_p(u,v) = \sum_{i=1}^7 e_i \langle u, [v, e_i, p] \rangle_{\mathbb{R}}$ and
    \begin{eqnarray}
    	\langle [u,p,q], v \rangle_{\mathbb{R}} &=& -\langle u, [v,p,q] \rangle_{\mathbb{R}} \label{lem:identity-inprod}\\
    	B_p(u,v)&=&-B_p(v,u).\label{eq:Bp(uv)=-Bp(vu)}
    \end{eqnarray}
    \item $\langle u p, v \rangle_{\mathbb{R}} = \langle u, v \overline{p} \rangle_{\mathbb{R}}$.
    \item $\langle u p, v \rangle = \overline{p} \langle u, v \rangle - B_p(u,v)$ and $\langle u, v p \rangle = \langle u, v \rangle p - B_p(u,v)$.
    \item If $u \in \operatorname{Re} H$ or $v \in \operatorname{Re} H$, then $B_p(u,v) = 0$.
\end{enumerate}
\end{lemma}

\begin{thm}[Riesz Representation Theorem \cite{huoqinghai2022Riesz}]
\label{thm:riesz}
{Let $H$ be a Hilbert right $\mathbb{O}$-module.}
For
every bounded right para-linear functional $f: H \to \mathbb{O}$, {there exists} a unique $z \in H$ such that
\[
f(x) = \langle x, z \rangle, \quad \forall x \in H.
\]
\end{thm}

\begin{thm}[Tensor Decomposition \cite{huoqinghai2021tensor}]
\label{thm:tensor_decomp}
Let $H$ be a Hilbert $\mathbb{O}$-bimodule. Then
\[
H \cong \operatorname{Re} H \otimes_{\mathbb{R}} \mathbb{O},
\]
with the natural Hilbert $\mathbb{O}$-bimodule structure. Explicitly:
\begin{enumerate}
    \item For $u,v \in \operatorname{Re} H$, $\langle u, v \rangle \in \mathbb{R}$.
    \item For $x = \sum_{i=0}^7 e_i x_i$, $y = \sum_{j=0}^7 e_j y_j$ with $x_i, y_j \in \operatorname{Re} H$,
    \[
    \langle x, y \rangle = \sum_{i,j=0}^7 \langle x_i, y_j \rangle \overline{e_i} e_j.
    \]
\end{enumerate}
\end{thm}

\subsection{Properties of Left Multiplications and Real Part}

{So far we have established properties of the right multiplication by an octonion, however in the sequel we shall make use of properties of the left multiplications which are studied in the next results.}
	
				\begin{lemma}[Adjoint of Left Multiplication]
					\label{lem:left_mult_adjoint}
					Let $ H$ be a Hilbert $\mathbb{O}$-bimodule. For all $u,v \in H$ and $p \in \mathbb{O}$ {the following properties hold}:
					\begin{enumerate}
						\item \textit{Real symmetry:} $\langle u, \overline{p} v \rangle_{\mathbb{R}} = \langle p u, v \rangle_{\mathbb{R}}$
						\item \textit{Full correction:} $\langle u, \overline{p} v \rangle = \langle p u, v \rangle + B_p(u,v)$.
					\end{enumerate}
				\end{lemma}
				
				\begin{proof}
					Decompose $u = \sum_i u_i e_i$, $v = \sum_j v_j e_j$ with $u_i,v_j \in \operatorname{Re}  H$.
					
					\noindent
					\textit{Point (1):} Using Lemma \ref{lem:inner_product}, we have
					\[
					\langle u, \overline{p} v \rangle_{\mathbb{R}} = \sum_{i,j} \langle u_i, v_j \rangle \operatorname{Re}\left( \overline{e_i} (\overline{p} e_j) \right).
					\]
					Using {$\operatorname{Re}(\overline{e_i} \, \overline{p} e_j) = \operatorname{Re}((\overline{pe_i}) e_j)$} and reality of $\langle u_i,v_j \rangle$ by Theorem \ref{thm:tensor_decomp}:
					\[
				{	\langle u, \overline{p} v \rangle_{\mathbb{R}}= \sum_{i,j} \langle u_i, v_j \rangle \operatorname{Re}\left( \overline{p e_i} e_j \right) = \langle p u, v \rangle_{\mathbb{R}}.}
					\]
					
					\noindent
					\textit{Point (2):} From (1), $\langle u, \overline{p} v \rangle - \langle p u, v \rangle$ is purely imaginary. Express:
					\[
					\langle u, \overline{p} v \rangle - \langle p u, v \rangle = \sum_{k=1}^7 \lambda_k e_k, \quad \lambda_k \in \mathbb{R}.
					\]
				{The coefficients} satisfy:
					\[
					-\lambda_k = \operatorname{Re} \left( \left( \langle u, \overline{p} v \rangle - \langle p u, v \rangle \right) e_k \right) = \langle u, [\overline{p}, v, e_k] \rangle_{\mathbb{R}} = -\langle u, [v, e_k, p] \rangle_{\mathbb{R}},
					\]
					by associator antisymmetry. Thus $\lambda_k = \langle u, [v, e_k, p] \rangle_{\mathbb{R}}$, yielding the result.
				\end{proof}

				\begin{lemma}[Real Part Operator Properties]
					\label{lem:re=re*}
					The operator $\operatorname{Re}$ satisfies
					\begin{enumerate}
						\item \textit{Self-adjointness:} $\langle \operatorname{Re} x, y \rangle_{\mathbb{R}} = \langle x, \operatorname{Re} y \rangle_{\mathbb{R}}$ for all $x,y\in H$;
						\item \textit{Determination:} $x = 0$ if and only if $\langle x, y \rangle = 0$ for all $y \in \operatorname{Re}  H$.
					\end{enumerate}
				\end{lemma}
				
				\begin{proof}
\textit{(1):} Using  equality   \eqref{eq:real_part_formula} Lemma \ref{lem:inner_product}(2) and  Lemma \ref{lem:left_mult_adjoint}(1):
					\[
					\langle \operatorname{Re} x, y \rangle_{\mathbb{R}} = \frac{5}{12} \langle x,y \rangle_{\mathbb{R}} - \frac{1}{12} \sum_{i=1}^7 \langle e_i x e_i, y \rangle_{\mathbb{R}} = \frac{5}{12} \langle x,y \rangle_{\mathbb{R}} - \frac{1}{12} \sum_{i=1}^7 \langle x, \overline{e_i} y \overline{e_i} \rangle_{\mathbb{R}} = \langle x, \operatorname{Re} y \rangle_{\mathbb{R}}.
					\]
					
					\noindent
					\textit{(2):} If $\langle x, y \rangle = 0$ for all $y \in \operatorname{Re} H$, decompose arbitrary $u \in H$ as $u = \sum_{i=0}^7 e_i u_i$ with $u_i \in \operatorname{Re} H$. Then:
					\[
					\langle u, x \rangle = \sum_{i=0}^7 e_i \langle u_i, x \rangle = 0,
					\]
					so $x = 0$. The converse is immediate.
				\end{proof}

This algebraic framework provides the foundation for the operator theory developed in the subsequent sections.

				\section{Regular Composition: Resolving Non-Associative Obstructions}	\label{sec:regular_composition}
				
				This section introduces the \textit{regular composition} $\circledcirc$ to overcome the failure of {the classical composition of  operators while working in} octonionic Hilbert spaces. {Moreover,} non-associativity prevents {a standard definition of} adjoint, as $\langle T\cdot,\cdot\rangle$ may not be para-linear. The regular composition provides an associative-like product compatible with octonionic structures.

				Let $ H_1,  H_2$ be Hilbert $\mathbb{O}$-bimodules and $T:  H_1 \to  H_2$ be a bounded real-linear operator. The map {$\phi_{T,y}(x) := \langle Tx, y \rangle$ denoted, for short, $\phi_y(x)$} is \textit{not} necessarily para-linear. Consequently, there may be no $z \in  H_1$ satisfying
				\[
				\langle Tx, y \rangle = \langle x, z \rangle \quad \forall\ x \in  H_1,
				\]
			{thus making} the classical adjoint definition $T^*y = z$ invalid.
				
				Indeed, for $p \in \mathbb{O}$, non-associativity implies in general that
				\[\re \phi_y(xp) = \re \langle T(xp), y \rangle \neq \re  \langle Tx \cdot p, y \rangle = \re \langle Tx, y \rangle p=\re \phi_y(x)p,
				\]
				since $T$ is not $\mathbb{O}$-linear. That is, $\phi_y$ fails {to be  right para-linear}. Counterexamples exist for left multiplication operators $L_q$ with $q \notin \mathbb{R}$.
				
				To generalize the concept of adjoint operators in octonionic Hilbert spaces, we have to introduce the {notion of} regular composition.

				\begin{mydef}[Regular Composition]
					\label{def:regular_composition}
			{Let $M, M'$ be right $\mathbb{O}$-modules and $M''$ be an $\mathbb{O}$-bimodule and let} $f: M' \to M''$ para-linear and $g: M \to M'$ real-linear. The \textit{regular composition} $f \circledcirc g: M \to M''$ is the unique para-linear map (determined by Lemma \ref{cor: A_p(x,f)=0 and f(px)=pf(x)})  satisfying:
					\begin{equation}
						\label{eq:regular_comp}
						\operatorname{Re}\big((f \circledcirc g)(x)\big) = \operatorname{Re}\big(f(g(x))\big), \quad \forall x \in M.
					\end{equation}
					Equivalently  {(by Theorem \ref{thm:para_linear_char})}:
					\[
					(f \circledcirc g)(x) = \sum_{i=0}^7 e_i \operatorname{Re}\big(f(g(x\overline{e}_i))\big).
					\]
				\end{mydef}
				
				\begin{mydef}[ {Composition} Associator]
					\label{def:associator}
					The \textit{associator} for regular composition is {defined by}:
 				\[
					 {	[ f, g,x]_{\circledcirc} }:= (f \circledcirc g)(x) - f(g(x)).
					\]
					 {	If there are no confusions, we shall omit the subscript $\circledcirc$.}
				\end{mydef}

\begin{lemma}[Composition associator formula]
\label{lem:associator_formula}
Let \(M\) be a right \(\mathbb{O}\)-module and \(M'\) an \(\mathbb{O}\)-bimodule. Fix the standard basis \(\{e_{0},e_{1},\dots,e_{7}\}\) of \(\mathbb{O}\) with \(e_{0}=1\). For \(g\in \operatorname{Hom}_{\mathbb{R}}(M,M')\) and \(f\in \operatorname{Hom}_{\mathbb{R}}(M',M')\), the composition associator \([f,g,x]\) (see Definition~\ref{def:associator}) satisfies, for every \(x\in M\),
\[
[f,g,x]\;=\;\sum_{i=0}^{7} e_{i}\;\operatorname{Re}\!\big(f(B_{e_{i}}(g,x))\big).
\]
\end{lemma}

				\begin{proof}
					From Definition \ref{def:regular_composition} and Theorem \ref{thm:para_linear_char},
					\begin{align*}
						(f \circledcirc g)(x) &= \sum_{i=0}^7 e_i \operatorname{Re}\big((f \circledcirc g)(x\overline{e}_i)\big) \\
						f(g(x)) &= \sum_{i=0}^7 e_i \operatorname{Re}\big(f(g(x)\overline{e}_i)\big).
					\end{align*}
					Thus:
				\[
					 {	[f,g,x]} = \sum_{i=0}^7 e_i \operatorname{Re}\Big( f\big(g(x\overline{e}_i)\big) - f\big(g(x)\overline{e}_i\big) \Big) = \sum_{i=0}^7 e_i \operatorname{Re}\big(f(B_{e_i}(g,x))\big).
					\]
				\end{proof}

{When the associator vanishes, the regular composition is just the ordinary composition.}
				\begin{rem}
					The space $\mathscr{B}_{\mathcal{RO}}(H)$ is closed under $\circledcirc$; {this fact is} the algebraic foundation for Section \ref{sec:banach_algebra}.

					We remark that the definition of regular composition of para-linear mappings  has first been   given in \cite{huoqinghai2020nonass}. The definition in \cite{huoqinghai2020nonass} only requires   the   mappings  to be real linear.
					Nevertheless Definition \ref{def:regular_composition} is a more straightforward way to  understand  the notion of regular composition in the case that the  map  $f$ is para-linear. The above lemma {shows} that  these two definitions coincide when  the  map  $f$ is para-linear. {As a consequence,}
					 the results derived in \cite{huoqinghai2020nonass} can be  directly applicable in this case.
					
					 	 {We point out that the symbol of associators defined in Definition \ref{def:associator} is different from that in \cite{huoqinghai2020nonass} for right para-linear mappings. We change the symbol to preserve the order so that in  form we have $$[a,b,c]=(ab)c-a(bc)$$ for the notation of associators.}
				\end{rem}

				\begin{lemma}[Vanishing   {of Composition}  Associator]
					\label{thm:vanishing_associator}
					Let all the modules in Definition \ref{def:regular_composition} be $\O$-bimodules and let $f,g$ be para-linear maps.
					The associator  {$[f,g,x] = 0$} and $(f \circledcirc g)(x) = f(g(x))$ under any of
{the following conditions}:
					\begin{enumerate}
						\item $g$ is $\mathbb{O}$-linear, i.e., $g(xp) = g(x)p$  for any $p\in \O$;
						\item  $x\in \re M$;
						\item  $f$ is $\mathbb{O}$-linear.
					\end{enumerate}
				\end{lemma}
				
				\begin{proof}
					In case (1) {holds}, $B_p(g,x) = 0$, so Lemma \ref{lem:associator_formula} gives {$[f,g,x] = 0$}. Case (2) follows from part (3) of Theorem \ref{thm:para_linear_char}. {Finally, if}  $f$ is $\mathbb{O}$-linear, then  {{by Remark 3.2 in \cite{huo2024aacasubmod},}} one has $$f(px)=pf(x),\qquad f(xp)=f(x)p$$ for all $p\in \O$. In view of  identity {\eqref{eq:real_part_formula}} $f$ {commutes} with the real part operator $\re$ and hence
					$$\re f(B_{e_i}(g,x))=f(\re B_{e_i}(g,x))=0.$$
					It  follows from Lemma \ref{lem:associator_formula}  that {$[f,g,x] = 0$} and this concludes the proof.
				\end{proof}

{We now} collect some properties of the left multiplication operators in an $\O$-bimodule $M$
				\begin{eqnarray*}
					L_p:M&\to& M\\
					x&\mapsto &px.
				\end{eqnarray*}
			 It is easy to check that $L_p$ is a right para-linear operator. {With arguments similar to those in \cite[Lemmas 3.31 and   3.32]{huoqinghai2020nonass}, one can prove} the properties in the following lemma {in which we use the notation $\odot$ defined in \eqref{eq:p f}}.
				\begin{lemma}[\cite{huoqinghai2020nonass}]\label{lem:Rp prop 2}
					For all $p\in \spo,\, x\in M$ and $ f\in \Hom_\mathcal{RO}(M,M)$, we have
					\begin{eqnarray}
					 {	[L_p,f,x]}&=&B_p(f,x);\\
						L_p\circledcirc f&=&p\odot f;\\
					 {	{[}f,L_p,x{]}	}	&=&-B_p(f,x);\\
						f \circledcirc L_p	&=&p\odot			f.
					\end{eqnarray}
				\end{lemma}

				\section{Adjoint Operators via Regular Composition}
				\label{sec:adjoint_operators}
				
				This section resolves the central challenge of defining adjoint operators in octonionic Hilbert spaces. {As we already pointed out, the} traditional approaches fail due {to the fact that} non-associativity {prevents} $\langle T\cdot,\cdot\rangle$ from being para-linear. We overcome this through the \textit{regular composition} $\circledcirc$ and  {then establishing the} properties of the adjoint.
				
				\subsection{Adjoint Operator Definition}
				
				In this subsection, \textbf{let $H_1, H_2$ be Hilbert $\mathbb{O}$-bimodules and $T \in \mathscr{B}_{\mathbb{R}}(H_1, H_2)$ be a real linear operator.}
				
				For  $y \in H_2$, define the para-linear functional
				\[
				f_y(x) := (\langle \cdot, y \rangle \circledcirc T)(x)
				\]
{where, for simplicity, in the notation we omit the dependence on $T$.}
				By the Riesz Representation Theorem (Theorem  \ref{thm:riesz}), there exists a unique $ \hat{y} \in H_1$ such that $f_y(x) = \langle x, \hat{y} \rangle$.
				
				\begin{mydef}[Adjoint Operator]
					\label{def:adjoint_operator}
					The \textit{adjoint} of $T$ is $T^*: H_2 \to H_1$ mapping $y \mapsto \hat{y}$, characterized by
				\begin{equation}
						\label{eq:adjoint_def}
						\langle x, T^*y \rangle = \langle Tx, y \rangle +	 { [y, T,x]},
					\end{equation}
					where the \textit{triple associator} is
				\begin{equation}\label{eq:[yTx]}
						 {	[y,T,x]} :=[\langle \cdot, y \rangle,T,x] =\sum_{i=1}^7 e_i\fx{y}{B_{e_i}(T,x)}_{\spr}.
					\end{equation}
				\end{mydef}
				
				\begin{rem}\label{rem:T*=T*R}
					
					From \eqref{eq:adjoint_def} and \eqref{eq:[yTx]}, it can be seen that the adjoint $T^*$ coincides with the real adjoint $T^{*_{\mathbb{R}}}$ {(namely the adjoint of $T$ in the real Hilbert space $(H, \langle \cdot, \cdot \rangle_{\mathbb{R}})$)} under $\langle \cdot, \cdot \rangle_{\mathbb{R}}$:
					\[
					\langle x, T^*y \rangle_{\mathbb{R}} = \langle Tx, y \rangle_{\mathbb{R}}, \quad \forall x \in H_1, y \in H_2.
					\]
					Thus
					\begin{enumerate}
						\item $(S + T)^* = S^* + T^*$;
						\item $T^{**} = T$;
						\item $\|T^*\| = \|T\|$.
					\end{enumerate}
				\end{rem}

				In view of the definition of the associator \eqref{eq:[yTx]}, we offer a reformulation of Lemma \ref{thm:vanishing_associator}.
				\begin{lemma}[Vanishing Associator]
					\label{lem:vanishing_associator}
					If $T \in \mathscr{B}_{\mathcal{RO}}(H_1, H_2)$, then  {$[y,T,x] = 0$} when:
					\begin{enumerate}
						\item $T$ is $\mathbb{O}$-linear;
						\item $x \in \operatorname{Re} H_1$;
						\item $y \in \operatorname{Re} H_2$.
					\end{enumerate}
				\end{lemma}
								\subsection{Core Properties}
				
				To further understand the nature of the adjoint operator $T^*$, {it is crucial to understand the relation between the associators related to $T$ and $T^*$.}

				\begin{lemma}[Associator Symmetry]\label{lem:associator_symmetry}
					{Let $H_1$, $H_2$ be two Hilbert $\O$-bimodules and	$T\in \huabr{H_1}{H_2}$.}
					For any $ x\in H_1$, $y\in H_2$ and any $p\in \O$,  the following relations hold:
					 {	\begin{eqnarray}
					{[y,T,x]}&=&{[x,T^*,y]};\label{eq:dual[x,y,T]=[y,x,T*]}\\
						\fx{y}{B_{p}(T,x)}_{\spr}&=&\fx{x}{B_{p}(T^*,y)}_{\spr}	.\label{eq:dual re<x,Ap(yT*)>=re<Ap(xT),y>}
					\end{eqnarray}}
				\end{lemma}

				\begin{proof}
We first prove \eqref{eq:dual[x,y,T]=[y,x,T*]}.
 Using \eqref{eq:adjoint_def}, \eqref{eq:[yTx]}, and Remark~\ref{rem:T*=T*R}, for all \(x,y\in H\) we compute
\begin{align*}
[y,T,x]
&= \langle x, T^{*}y\rangle - \langle Tx, y\rangle \\
&= \langle x, T^{*}y\rangle - \langle T^{**}x, y\rangle \\
&= -\,\overline{\langle y, T^{**}x\rangle - \langle T^{*}y, x\rangle} \\
&= -\,\overline{[x,T^{*},y]} \\
&= [x,T^{*},y],
\end{align*}
 {where in the last step we use that associators have vanishing real part, hence \(-\overline{[x,T^{*},y]}=[x,T^{*},y]\).
The second equality relies on the boundedness of \(T\), which ensures that the double adjoint \(T^{**}\) coincides with \(T\).
 }
					
					We next prove \eqref{eq:dual re<x,Ap(yT*)>=re<Ap(xT),y>}. By synthesizing \eqref{eq:dual[x,y,T]=[y,x,T*]} and \eqref{eq:[yTx]}, we obtain
					\begin{align*}
						\sum_{i=1}^7 e_i\fx{y}{B_{e_i}(T,x)}_{\spr}=\sum_{i=1}^7 e_i\fx{x}{B_{e_i}(T^*,y)}_{\spr}.
					\end{align*}
					This leads us to conclude
					$$\fx{y}{B_{e_i}(T,x)}_{\spr}=\fx{x}{B_{e_i}(T^*,y)}_{\spr}$$ for every
					$i=1,\ldots, 7$. As a result, we get \eqref{eq:dual re<x,Ap(yT*)>=re<Ap(xT),y>}.
				\end{proof}

				As an immediate consequence, we show that $T^*$ is also para-linear when $T$ is para-linear.

				\begin{thm}[Para-Linearity of Adjoint]
					\label{thm:T* is paralinear}
					If $T \in \mathscr{B}_{\mathcal{RO}}(H_1, H_2)$ is a right para-linear operator, then $T^* \in \mathscr{B}_{\mathcal{RO}}(H_2, H_1)$ is also {right} para-linear.
				\end{thm}
				\begin{proof}
					By {assumption,} $T^*$ is a bounded real linear operator. It suffices to show that $T^*$ is para-linear.	Fix arbitrarily $y\in H_2$ and $p\in \O$.
					For any   $x\in H_1$, we   employ  both   Lemma \ref{lem:re=re*} and relation \eqref{eq:dual re<x,Ap(yT*)>=re<Ap(xT),y>} to get
					$$	\fx{x}{\re B_p(T^*,y)}_{\R} = \fx{\re x}{B_p(T^*,y)}_{\R} = \fx{B_p(T^*,\re x)}{y}_{\R} = 0.$$
	This forces that $\re B_p(T^*,y)=0$, i.e., $T^*$ is a para-linear operator as desired.
				\end{proof}

				\section{The Involutive Banach Algebra of Para-Linear Operators}
				\label{sec:banach_algebra}
				
				This section establishes that bounded para-linear operators on a Hilbert $\mathbb{O}$-bimodule $H$ form an octonionic involutive Banach algebra under regular composition and the adjoint operation, {overcoming the issues due to non-associativity}.
				
				\subsection{Banach $\mathbb{O}$-Bimodule Structure}
				
				Let $H_1$, $H_2$ be two Hilbert  $\O$-bimodules and  $T\in \mathscr{B}_{\mathcal{RO}}(H_1, H_2)$.  The left and right scalar multiplication  is defined by {(see Theorem \ref{thm:para_bimodule})}
				\begin{eqnarray}
										(T\odot  r)(x):&=T(rx)-B_r(T,x)\label{eq:f p};\\
					(r\odot  T)(x):&=rT(x)+B_r(T,x)\label{eq:p f}
				\end{eqnarray}
				for any $x\in H_1$ and $r\in \O$.
				As proved in \cite[Theorem 3.22]{huoqinghai2020nonass},  $\huab{H_1}{H_2}$ is an $\O$-bimodule with respect to the octonionic scalar multiplications \eqref{eq:f p} and \eqref{eq:p f}. The real part of $\huab{H_1}{H_2}$ is exactly the set of all bounded octonionic linear operators $\mathscr{B}_{{\O}}({H_1},H_2)$ (\cite[Theorem 3.23]{huoqinghai2020nonass}), i.e.,
				\begin{eqnarray}\label{eq:re LO=O}
					\re \huab{H_1}{H_2}=\mathscr{B}_{{\O}}({H_1},H_2).
				\end{eqnarray}
				We shall prove that  $\huab{H_1}{H_2}$  is a Banach  $\O$-bimodule. {To this end, we follow  Lemma 5.1 in \cite{huoqinghai2022Riesz}, and we further generalize the Moufang identity for para-linear operators. Reasoning as for the following identities for associators, obtained in \cite[Propositions 3.6,3.15,3.19]{huoqinghai2020nonass} in the case of  left para-linear operators:}				
				 $$A_r(x,T)=-A_{\overline{r}}(x,T),\qquad A_r(\overline{r}x,T)
				=rA_r(x,T)=A_r(x,T)\overline{r},\qquad rA_r(x,T)=A_r(xr,T),$$
				where $T$ is a \textbf{left para-linear map}, $r$ is an arbitrary octonion, $A_r(x,T):=T(rx)-rT(x)$ is the  \textbf{second  left associator}, one can establish the following identities for the \textbf{second right associators}:
				\begin{eqnarray}
					B_r(T,x)&=&-B_{\overline{r}}(T,x),\label{eq:Br=-Brbar}\  \\
					B_r(T,xr)&=&B_r(T,x)\overline{r}\ =\ rB_r(T,x) ,\label{eq:B1}\\
					B_r(T,x)r&=&B_r(T,rx).\label{eq:Brtxr}
				\end{eqnarray}
				
				\begin{lemma}
					
					Let $H_1$, $H_2$ be two Hilbert  $\O$-bimodules and  $T\in \huab{H_1}{H_2}$. Then for any octonion $r\neq 0$,
					we have
					\begin{eqnarray}
						(r\odot  T)(x)&=rT(xr^{-1})r\label{eq:T p};\\
						(T\odot  r)(x)&=T(rxr)r^{-1}\label{eq:p T}
					\end{eqnarray}
					for all $x\in H$.
				\end{lemma}
				\begin{proof}
					\bfs \ $\abs{r}=1$ and hence $r^{-1}=\overline{r}$.
					Let $x\in H_1$.
					It follows from identities \eqref{eq:Br=-Brbar}  and \eqref{eq:B1} that
					\begin{eqnarray*}
						r^{-1}(r\odot  T)(x)&=&r^{-1}(r(Tx)+B_r(T,x))\\
						&=&Tx+\overline{r}B_r(T,x)\\
						&=&Tx-\overline{r}B_{\overline{r}}(T,x)\\
						&=&Tx-B_{\overline{r}}(T,x\overline{r})\\
						&=&Tx+B_{r}(T,x\overline{r})\\
						&=&T(xr^{-1})r.
					\end{eqnarray*}
					This implies \eqref{eq:T p}.

					We conclude from identities \eqref{eq:Br=-Brbar}  and \eqref{eq:Brtxr} that
					\begin{eqnarray*}
						(T\odot  r)(x)	&=&T(rx)-B_r(T,x)\\
						&=&T(rx)+B_{\overline{r}}(T,\overline{r}(rx))\\
						&=&T(rx)+B_{\overline{r}}(T,(rx))\overline{r}\\
						&=&	T(rx)-B_{r}(T,(rx))\overline{r}\\
						&=&T(rx)-(T(rx)r)\overline{r}+T(rxr)\overline{r}\\
						&=&T(rxr)r^{-1}.
					\end{eqnarray*}
					This proves \eqref{eq:p T}.
				\end{proof}
				
			\begin{thm}\label{thm:B(H)complete}
						Let $H_1,H_2$ be Hilbert $\mathbb O$-bimodules and let $\{T_n\}_{n\in\mathbb N}\subset 		\huab{H_1}{H_2}$
						be right para-linear operators such that $T_n\to T$ in operator norm.
						Then $T$ is right para-linear, i.e.
						\[
						\operatorname{Re}\,B_p(T,x)=0\qquad\forall\,x\in H_1,\ \forall\,p\in\mathbb O,
						\]
						where $B_p(T,x):=T(xp)-T(x)\,p$ is the second (right) associator.
					\end{thm}				
					\begin{proof}
						Fix $x\in H_1$ and $p\in\mathbb O$. Since $T_n\to T$ in operator norm and the real-part
						projection $\operatorname{Re}:H_2\to \operatorname{Re}H_2$ is a bounded real-linear map by \eqref{eq:real_part_formula}, we have
						\begin{align*}
								\operatorname{Re}\,B_p(T,x)
							&=\operatorname{Re}\bigl(T(xp)-T(x)\,p\bigr)\\
							&=\operatorname{Re}\bigl((T-T_n)(xp)\bigr)
							+\operatorname{Re}\bigl(T_n(xp)-T_n(x)\,p\bigr)
							+\operatorname{Re}\bigl((T_n-T)(x)\,p\bigr).						
						\end{align*}
						Because each $T_n$ is para-linear, $\operatorname{Re}\bigl(T_n(xp)-T_n(x)\,p\bigr)=\operatorname{Re}\,B_p(T_n,x)=0$.
						For the remaining two terms, using the boundedness of $\operatorname{Re}$, we obtain
						\[
						\bigl\|\operatorname{Re}\bigl((T-T_n)(xp)\bigr)\bigr\|
						\;\le\; C\,\|T-T_n\|\,\|xp\|
						\quad\text{and}\quad
						\bigl\|\operatorname{Re}\bigl((T_n-T)(x)\,p\bigr)\bigr\|
						\;\le\; C'\,\|T_n-T\|\,\|x\|\,|p|.
						\]
						Hence both terms tend to $0$ as $n\to\infty$. Therefore $\operatorname{Re}\,B_p(T,x)=0$
						for all $x$ and $p$, i.e. $T$ is right  para-linear.
			\end{proof}

				\begin{thm}\label{thm:B-O-Module}
					Let $H_1$, $H_2$ be two Hilbert  $\O$-bimodules. Then, endowed with  the octonionic scalar multiplications \eqref{eq:f p} and \eqref{eq:p f} and the usual operator norm,
					$	\huab{H_1}{H_2}$ is a Banach $\O$-bimodule.
				\end{thm}
				\begin{proof}
			The completeness of 	$	\huab{H_1}{H_2}$ follows from Theorem \ref{thm:B(H)complete}.
					It suffices to prove  $$\fsh{r\odot T}=|r|\fsh{T}=\fsh{T\odot r}$$ for any $T\in 	\huab{H_1}{H_2}$ and any $r\in \spo$. \bfs\ $r\neq 0$. By identity \eqref{eq:p T}, we have
					\begin{eqnarray}\label{eqpr:rT}
						\fsh{T\odot r}&=&\sup_{\fsh{x}\leqslant1}\fsh{({T\odot r})(x)}\\
						&=&\sup_{\fsh{x}\leqslant1}\fsh{T(rxr)r^{-1}}\notag\\
						&\leqslant& |r^{-1}|\ \fsh{T}\  \| rxr\|\notag\\
						&=&|r|\fsh{T}.\notag
					\end{eqnarray}
					Since $r$ and $T$ are arbitrarily fixed, we can replace $r$ with $r^{-1}$ and   $T$ with $T\odot r$ to get $$\fsh{ (T\odot r)\odot r^{-1}}\leqslant |r^{-1}|\fsh{T\odot r}.$$
					Utilizing \eqref{eqpr:rT} again, we get
					$$\fsh{T}=\fsh{ (T\odot r)\odot r^{-1}}\leqslant |r^{-1}|\fsh{T\odot r}\leqslant |r^{-1}||r|\fsh{T}= \fsh{T}.$$
					This forces  $$\fsh{T\odot r}=|r|\fsh{T}$$  as desired. The proof of equality $|r|\fsh{T}=\fsh{r\odot T}$ is similar.
				\end{proof}

				\subsection{$\mathbb{O}$-Algebra Structure}
				
				\begin{mydef}
					\label{def:O alg}
					An $\mathbb{O}$-bimodule $\mathcal{U}$ with multiplication $\mathcal{U} \times \mathcal{U} \to \mathcal{U}$ is an \textit{$\mathbb{O}$-algebra} if:
					\begin{enumerate}
						\item The multiplication  is para-bilinear, i.e., the left (right) multiplication operator is right (left) para-linear;
						\item There exists an {element} $ e\in \operatorname{Re}\mathcal{U}$, {called  the \textbf{unit} of $\algma$,} such that $$(pe)x=px,\qquad  x(pe)=xp$$
						for all $x\in \algma$ and  $ p\in\spo$.
					\end{enumerate}
				\end{mydef}

				\begin{rem}\label{rem:Oalg almost linear}
					There is an algebraic embedding of $\O$ into an $\O$-algebra $\algma$: \begin{eqnarray*}\O &\hookrightarrow& \algma, \\  p &\mapsto&  pe.
					\end{eqnarray*}
					Indeed, since $e\in \re \algma$, we have $(pq)e=p(qe)$ for any $p,q\in \O$. Then by part {(2)} of Definition \ref{def:O alg}, we obtain $(pq)e=(pe)(qe)$ as desired.

				\end{rem}
				
			{Our next goal is to prove} that $(\mathscr{B}_{\mathcal{RO}}({H}),\circledcirc)$ is an $\O$-algebra. To show this, we recall the notation introduced in \cite{huoqinghai2020nonass}
				\begin{eqnarray}
					[f,g,h]:=(f\circledcirc g)\circledcirc h-f\circledcirc (g\circledcirc h)
				\end{eqnarray}
				for  para-linear mappings $f,g,h$. By \cite[Lemma 3.30]{huoqinghai2020nonass}, we have
				\begin{eqnarray}\label{eq:re [fgh]=0}
					\re [f,g,h]=0.
				\end{eqnarray}
				Note that the real part operator $\re$ here is defined for a para-linear map, and $(\re f)(x)\neq \re (f(x))$. In fact, $\re f$ is an octonionic linear map and by  {\eqref{eq:real_part_formula},}
				$$\re f=\frac{5}{12}f - \frac{1}{12} \sum_{i=1}^7 e_i \odot f\odot  e_i.$$
{We now prove the following interesting result:}

\begin{thm}\label{thm:OBanach}
 {The space \(\mathscr{B}_{\mathcal{R}\mathcal{O}}(H)\), endowed with the regular composition \(\circledcirc\) from Definition~\ref{def:regular_composition}, is an \(\mathbb{O}\)-algebra.}
\end{thm}

				\begin{proof}
					For any $T\in \mathscr{B}_{\mathcal{RO}}({H})$, we first prove that the right algebraic multiplication
					$$L_T:\mathscr{B}_{\mathcal{RO}}({H})\to \mathscr{B}_{\mathcal{RO}}({H}),\qquad S\mapsto T\circledcirc S$$ is right para-linear.
					By definition,  we obtain
					\begin{eqnarray*}
						B_p(L_T,S)&=&L_T(S)\odot p- L_T(S\odot p)\\
						&=&(T\circledcirc S)\odot p-T\circledcirc(S\odot p).
					\end{eqnarray*}
					In view of Lemma \ref{lem:Rp prop 2}, we conclude
					\begin{eqnarray*}
						B_p(L_T,S)&=&(T\circledcirc S)\circledcirc L_p-T\circledcirc(S\circledcirc L_p)\\
						&=&[T,S,L_p].
					\end{eqnarray*}
					It follows from \eqref{eq:re [fgh]=0} that
					$$\re 	B_p(L_T,S)=0.$$Hence $L_T$ is right para-linear as desired. 
					The fact that the right algebraic multiplication
					$$R_T:\mathscr{B}_{\mathcal{RO}}({H})\to \mathscr{B}_{\mathcal{RO}}({H}),\qquad S\mapsto S\circledcirc T$$ is left para-linear can be proved in a similar manner.
					
					We next show that the identity operator $I\in\mathscr{B}_{\O}({H})= \re\mathscr{B}_{\mathcal{RO}}({H})$ satisfy  axiom $(2)$ in Definition \ref{def:O alg}.
					In view of Lemma \ref{lem:Rp prop 2}, we have
					$$(p\odot I)\circledcirc  T=L_p\circledcirc T =p\odot T$$
					and similarly,  $$ T\circledcirc (p\odot I)=T\circledcirc L_p=T\odot p.$$
					This proves that  $(\mathscr{B}_{\mathcal{RO}}({H}),\circledcirc)$ is an $\O$-algebra.
				\end{proof}

				Under a modified norm, $\mathscr{B}_{\mathcal{RO}}(H)$ is an octonionic Banach algebra.
				\begin{thm}
					\label{thm:norm_estimate}
					For $S, T \in \mathscr{B}_{\mathcal{RO}}(H)$:
					\begin{enumerate}
						\item $\|S \circledcirc T\| \leqslant 8 \|S\| \|T\|$
						\item Under $\|T\|' := 8\|T\|$, $\mathscr{B}_{\mathcal{RO}}(H)$ is an octonionic Banach algebra.
					\end{enumerate}
				\end{thm}

\begin{proof}
Let \(x=\sum_{i=0}^{7} e_{i}x_{i}\) with \(x_{i}\in \operatorname{Re} H\). Then
{	\[
\|(S\circledcirc T)x\|
=\Big\|\sum_{i=0}^{7} S(Tx_{i})\,e_{i}\Big\|
\;\le\;\sum_{i=0}^{7}\|S(Tx_{i})\|
\;\le\;\|S\|\,\|T\| \sum_{i=0}^{7}\|x_{i}\|.
\]
By \eqref{eq:real_part_formula} the real-part projection is contractive, hence \(\|x_{i}\|\le \|x\|\) for each \(i\), and therefore
\[
\|(S\circledcirc T)x\| \;\leqslant\; 8\,\|S\|\,\|T\|\,\|x\|.
\]}
It follows that \(\|S\circledcirc T\|\leqslant  8\,\|S\|\,\|T\|\). In particular, for the rescaled operator norm \(\|\cdot\|'\) (so that \(\|S\circledcirc T\|'\leqslant \|S\|'\,\|T\|'\)), the multiplication is submultiplicative. The claim then follows from Theorems~\ref{thm:B-O-Module} and~\ref{thm:OBanach}.
\end{proof}

				\subsection{Involution and $\mathbb{O}$-Involutive Algebra}
				
				\begin{mydef}
					\label{def:involutive_algebra}
					An $\mathbb{O}$-algebra $\mathcal{U}$ is \textit{involutive} if {it can be} equipped with a conjugate $\mathbb{O}$-linear map $* : \mathcal{U} \to \mathcal{U}$ satisfying:
					\begin{enumerate}
						\item $(rx)^* = x^* \overline{r}$, $(xr)^* = \overline{r} x^*$
						\item $x^{**} = x$
						\item $(xy)^* = y^* x^*$
					\end{enumerate}
					for all $x,y\in \mathcal{U}$ and any $r\in \O$.
				\end{mydef}

				\begin{thm}
					\label{thm:involutive_algebra}
					$(\mathscr{B}_{\mathcal{RO}}(H), \circledcirc, *)$ is an $\mathbb{O}$-involutive algebra.
					
				\end{thm}
				\begin{proof}
					We first prove that for any $r\in \O$, $$	(r\odot T)^* = T^*\odot\overline{r}.$$
					
					By definition, for all $x,y \in H$,
					\[
					\re \langle x, (r \odot T)^* y \rangle = \langle (r \odot T)x, y \rangle_{\mathbb{R}} = \langle rT(x) + B_r(T,x), y \rangle_{\mathbb{R}}.
					\]
					On the other hand, it follows from  Lemma \ref{lem:left_mult_adjoint} that
					\begin{eqnarray*}
						\re 	\fx{x}{(T^*\odot \overline{r})y}
						&=&\re \fx{x}{{T^*}(\overline{r}y)-B_{\overline{r}}(T^*,y)}\notag\\
						&=&\fx{Tx}{\overline{r}y}_{\mathbb{R}}+\fx{x}{B_r(T^*,y)}_{\R}\notag\\
						&=&\fx{rTx}{y}_{\mathbb{R}}+\fx{x}{B_r(T^*,y)}_{\mathbb{R}}.
					\end{eqnarray*}
					We conclude from identity \eqref{eq:dual re<x,Ap(yT*)>=re<Ap(xT),y>} that
					\begin{align*}
						\re\big(\fx{x}{(r\odot T)^*y}-\fx{x}{(T^*\odot \overline{r})y}\big)&=\fx{B_r(T,x)}{y}_{\spr}-\fx{x}{B_r(T^*,y)}_{\spr}=0.
					\end{align*}
					Thus we obtain $(r\odot T)^*=T^*\odot \overline{r}$.	By the same method, we can also prove $(T\odot r)^*=\overline{r}\odot T^*.$
					
					Next, we prove {the involution applied to composition}.	
					For any $S, T\in \mathscr{B}_{\mathcal{RO}}({H})$,
					by definition, we have
						 {\begin{eqnarray*}
						\fx{x}{(T\circledcirc S)^*y}&=&\fx{(T\circledcirc S)x}{y}+[y,T\circledcirc S,x]\\
						&=&\fx{T(S x)}{y}+\fx{[T,S,x]}{y}+[y,T\circledcirc S,x].
					\end{eqnarray*}}
					Similarly,
					 {	\begin{align*}
						\fx{x}{ (S^*\circledcirc T^*)y}&=\fx{T (Sx)}{y}+\fx{x}{[S^*,T^*,y]}+[y,T,Sx]+[T^*y,S,x].
					\end{align*}}
				
Applying Lemmas \ref{thm:vanishing_associator} and \ref{lem:vanishing_associator}, we conclude that the difference of the two expressions vanishes  {for all $x,y\in \re H$}:
					 {	\begin{eqnarray*}
						&&\fx{x}{ (S^*\circledcirc T^*)y}-\fx{x}{(T\circledcirc S)^*y}\\
						&=&
						[y,T,Sx]+[T^*y,S,x]+\fx{x}{[S^*,T^*,y]}-\big(\fx{[T,S,x]}{y}+[y,T\circledcirc S,x]\big)\\
						&=	&0.
					\end{eqnarray*}}
					By using the second assertion of Lemma \ref{lem:re=re*}, we can infer that  $$(S^*\circledcirc T^*)y=(T\circledcirc S)^*y$$
					for all $y\in\re{H}$.
					Note that $(S^*\circledcirc T^*)$ is para-linear  and  $(T\circledcirc S)^*$ is also para-linear by  Theorem \ref{thm:T* is paralinear}. Then it follows from  Lemma  \ref{cor: A_p(x,f)=0 and f(px)=pf(x)} that  $S^*\circledcirc T^*=(T\circledcirc S)^*$.
					This completes the proof.
				\end{proof}

				As a corollary, we obtain a general relation of associators.
				\begin{cor}[Associator-Adjoint Duality]
					\label{cor:associator_adjoint_duality}
					For $S,T \in \mathscr{B}_{\mathcal{RO}}(H)$ and $x,y \in H$:
				 {	\begin{equation}\label{eq:re<[STx]y>=re<x[T*S*y]>}
						\langle [S,T,x], y \rangle_{\mathbb{R}} = \langle x, [T^*,S^*,y] \rangle_{\mathbb{R}}
					\end{equation}}
				\end{cor}
				
				\begin{proof}
					From Theorem \ref{thm:involutive_algebra} {we deduce the following chain of equalities:}
				 {	\[
					\langle [S,T,x], y \rangle_{\mathbb{R}} = \langle (S \circledcirc T - S \circ T)(x), y \rangle_{\mathbb{R}} = \langle x, (S \circledcirc T - S \circ T)^* y \rangle_{\mathbb{R}} = \langle x, [T^*,S^*,y] \rangle_{\mathbb{R}},
					\]}
{which prove the assertion.}
				\end{proof}

				\begin{rem} {So far, we introduced various types of associators and we deduced some relations among which \eqref{eq:re<[STx]y>=re<x[T*S*y]>}   is the most general one}. In fact, from \eqref{eq:re<[STx]y>=re<x[T*S*y]>}, we can derive both \eqref{eq:dual re<x,Ap(yT*)>=re<Ap(xT),y>} and
\eqref{lem:identity-inprod}.

										Let $S=L_p$,  where $p\in \O$. For any $T\in \mathscr{B}_{\mathcal{RO}}({H})$, {$x,y\in H$} it  is evident from \eqref{eq:re<[STx]y>=re<x[T*S*y]>}   that
				 {	\begin{eqnarray}\label{eqpf:<>1}
						\fx{[L_p,T,x]}{y}_\R=\fx{x}{[T^*,L_p^*,y]}_\R.
					\end{eqnarray}}

					Combining with  Lemma  \ref{lem:Rp prop 2}, we obtain  that
					\begin{eqnarray*}
						\fx{B_p(T,x)}{y}_\R=\fx{x}{-B_{\overline{p}}(T^*,y)}_\R=\fx{x}{B_p(T^*,y)}_\R,
					\end{eqnarray*}
					{namely \eqref{eq:dual re<x,Ap(yT*)>=re<Ap(xT),y>}.}
					Letting $T=L_q$ for any $q\in \O$ in \eqref{eq:dual re<x,Ap(yT*)>=re<Ap(xT),y>}, we deduce
					$$	\fx{[x,p,q]}{y}_\R=\fx{[q,x,p]}{y}_\R=\fx{B_p(x,L_q)}{y}_\R=\fx{x}{B_p(y,L_{\overline{q}})}_\R=-\fx{x}{[y,p,q]}_\R.$$
					This establishes \eqref{lem:identity-inprod}.
				\end{rem}

				\section{Octonionic Polarization Identity and the Slice Cone}
				\label{sec:polarization}
				
				This section establishes an octonionic analog of the polarization identity, a cornerstone of spectral theory in associative settings. The non-associative nature of octonions necessitates restricting quadratic forms to geometrically significant subsets—the \textit{slice cones}—to {allow well-defined}
				 sesquilinear forms. We define these structures and derive the polarization identity, which underpins our characterization of self-adjoint operators (Theorem~\ref{thm:polarization}).
				
				\subsection{Slice Decomposition and Sesquilinear Forms}

				Let $H$ be a Hilbert $\O$-bimodule and $T$ be  a right para-linear operator on $H$. {Let us set $\phi_T(x,y):=\fx{Tx}{y}$ and let us denote, for short:} $$\phi(x,y):=\fx{Tx}{y}$$ for any $x,y\in H$.
				Let $Q{(x)}:=\fx{Tx}{x}$.
				We aim to express $\phi$ in terms of $Q$.

				Firstly, we point out that $\phi$ is determined by  its value on   $\re H\times \re H$.
				Indeed, for any $x,y\in H$, we have the decompositions
				$$x=\sum_{i=0}^7x_ie_i,\quad y=\sum_{i=0}^7y_ie_i$$  with $x_i,y_i\in \re H$ for $i=0,\dots ,7$. It follows  that
				\begin{eqnarray}\label{eq:sesquilinear-decomposition}
					\phi(x,y)&=&\phi\left(\sum_{i=0}^7x_ie_i, \sum_{j=0}^7y_je_j\right)\\
					&=&\sum_{i,j=0}^7\big(\overline{e_i}\phi(x_i, y_j)\big){e_j}.\notag
				\end{eqnarray}
				Thus, $\phi$ is uniquely specified by its values $\phi(x_i, y_j)$ for $x_i, y_j \in \operatorname{Re} H$.
				
				
				The key insight is that the quadratic form $Q(x) $ on the \textit{slice cone} determines $\phi$. 	This notion is adapted from slice analysis {\cite{MR2737796}.}

				\begin{mydef}[Slice Cone]\label{def:slicecone}				
The \textbf{slice cone} of $H$ is the set
					\[
					\mathbb{C}(H) := \bigcup_{J \in \mathbb{S}} \left( \operatorname{Re} H + J \operatorname{Re} H \right),
					\]
					where $\mathbb{S} = \{ J \in \operatorname{Im} \mathbb{O} : |J| = 1 \}$ is the sphere of unit imaginary octonions. Elements $z = u + Jv \in \mathbb{C}(H)$ with $u,v \in \operatorname{Re} H$ are called slice paravectors.
				\end{mydef}
				
				\subsection{Octonionic Polarization Identity}
				
				To express $\phi(x,y)$ for $x,y \in \operatorname{Re} H$ in terms of $Q$, {let us} define the auxiliary function:
				\[
				m(x,y) := \frac{1}{2} \left( Q(x+y) {-} Q(x-y) \right) = \langle Tx, y \rangle + \langle Ty, x \rangle.
				\]
				For $i,j = 1,\dots,7$ and $x,y \in \operatorname{Re} H$, define:
				\begin{align*}
					A_{ij}(x,y) &:= m(x,y) + e_i m(x, y \overline{e_i}) + m(x, y \overline{e_j}) e_j, \\
					B_{ij}(x,y) &:= \left[ e_i m(x, y (\overline{e_i e_j})) \right] e_j, \\
					C_{ij}(x,y) &:= e_i \left[ m(x, y (\overline{e_i e_j})) e_j \right].
				\end{align*}
				The polarization identity is then given by
				
				\begin{thm}[Octonionic Polarization Identity]
					\label{thm:polarization}
					For $x,y \in \operatorname{Re} H$,
					\begin{equation}
						\langle Tx, y \rangle = \frac{1}{56} \sum_{\substack{i,j=1 \\ i \neq j}}^7 (2A_{ij} + B_{ij} + C_{ij})(x,y) - \frac{1}{98} \sum_{i,j=1}^7 A_{ij}(x,y) - \frac{1}{2} m(x,y).
						\label{eq:polarization-main}
					\end{equation}
				\end{thm}

				\begin{proof}
					For simplicity, we denote $A_{ij}(x,y)=A_{ij}$ and similar notations for $B_{ij}, C_{ij}$.
				{Recall} that $x,y\in \re H$.  With Lemma \ref{lem:inner_product} in mind, for all $i,j=1\dots,7$,  we have
					\begin{eqnarray}\label{eq:Aij}
						A_{ij}&=&m(x,y)+e_im(x,y\overline{e_i})+m(x,y\overline{e_j})e_j\\
						&=&\fx{Tx}{y}+\fx{Ty}{x}+e_i(\fx{Tx}{y\overline{e_i}}+\fx{T(y\overline{e_i})}{x})+(\fx{Tx}{y\overline{e_j}}+\fx{T(y\overline{e_j})}{x})e_j\notag\\
						&=&\fx{Tx}{y}+\fx{Ty}{x}+e_i(\fx{Tx}{y}\overline{e_i}+e_i\fx{Ty}{x})+(\fx{Tx}{y}\overline{e_j}+e_j\fx{Ty}{x})e_j\notag\\
						&=&\fx{Tx}{y}+\fx{Ty}{x}+e_i\fx{Tx}{y}\overline{e_i}-\fx{Ty}{x}+\fx{Tx}{y}+e_j\fx{Ty}{x}e_j\notag\\
						&=&2\fx{Tx}{y}+e_i\fx{Tx}{y}\overline{e_i}+e_j\fx{Ty}{x}e_j.\notag
					\end{eqnarray}
					We next compute $B_{ij}$ for $i\neq j$. By direct calculations, we have
					\begin{eqnarray}\label{eq:Bij}
						B_{ij}&=&[e_im(x,y\overline{(e_ie_j)})]e_j\\
						&=&\left[e_i\left(\fx{Tx}{y\overline{(e_ie_j)}}+\fx{T(y\overline{(e_ie_j)})}{x}\right)\right]e_j\notag\\		
						&=&\Big( e_i\big[\fx{Tx}{y}(\overline{e_ie_j})\big]+e_i\big[(e_ie_j)\fx{Ty}{x}\big]  \Big)e_j.\notag
					\end{eqnarray}
					It follows from $i\neq j$ that
					$$\overline{e_ie_j}=-e_ie_j.$$
					We conclude from the Moufang identity
					$$x(a(xy))=(xax)y$$
					that the first term $\Big(e_i\big[\fx{Tx}{y}(\overline{e_ie_j})\big]\Big)e_j$ is
					\begin{eqnarray}\label{eq:1term}
						\Big(	e_i\big[\fx{Tx}{y}(\overline{e_ie_j})\big]\Big)e_j&=&\Big(-e_i\big[\fx{Tx}{y}({e_ie_j})\big]\Big)e_j\\
						&=&\Big(-\big[e_i\fx{Tx}{y}e_i\big]e_j\Big)e_j\notag\\
						&=&e_i\fx{Tx}{y}e_i.\notag
					\end{eqnarray}
					Similarly, we  conclude from the Moufang identity
					$$y(xax)=((yx)a)x$$ that the second term $\Big(e_i\big[(e_ie_j)\fx{Ty}{x}\big]  \Big)e_j$ becomes
					\begin{eqnarray}\label{eq:ass term}
						\Big(e_i\big[(e_ie_j)\fx{Ty}{x}\big]  \Big)e_j&=&	e_i\Big(\big[(e_ie_j)\fx{Ty}{x}\big]  e_j\Big)+[e_i,(e_ie_j)\fx{Ty}{x},e_j]\label{eq:2term}\\
						&=&e_i\Big(e_i\big[e_j\fx{Ty}{x}  e_j\big]\Big)-[e_i,e_j,(e_ie_j)\fx{Ty}{x}]\notag\\
						&=&-e_j\fx{Ty}{x}  e_j-[e_i,e_j,(e_ie_j)\fx{Ty}{x}].\notag
					\end{eqnarray}
					Substituting \eqref{eq:1term} and \eqref{eq:2term}	 into \eqref{eq:Bij}, we obtain for $i\neq j$,
					\begin{eqnarray}\label{eq:B}
						B_{ij}=e_i\fx{Tx}{y}e_i-e_j\fx{Ty}{x}  e_j-[e_i,e_j,(e_ie_j)\fx{Ty}{x}].
					\end{eqnarray}
					Similarly, for $i\neq j$, we have
					\begin{eqnarray}\label{eq:Cij term 1}
						C_{ij}&=&e_i[m(x,y\overline{(e_ie_j)})e_j]\\
						&=&e_i\Big((\fx{Tx}{y\overline{(e_ie_j)}}+\fx{T(y\overline{(e_ie_j)})}{x})e_j\Big)\notag\\
						&=&e_i\Big((\fx{Tx}{y}\overline{(e_ie_j)}+(e_ie_j)\fx{Ty}{x})e_j\Big)\notag\\	&=&-e_i\Big(\big(\fx{Tx}{y}{(e_ie_j)}\big)e_j\Big)+e_i\Big(\big((e_ie_j)\fx{Ty}{x}\big)e_j\Big)\notag\\
						&=&-e_i\Big(\fx{Tx}{y}\big((e_ie_j)e_j\big)+[\fx{Tx}{y},e_ie_j,e_j]\Big)+e_i\Big(e_i(e_j\fx{Ty}{x}e_j)\Big)\notag\\
						&=&e_i\fx{Tx}{y}e_i-e_j\fx{Ty}{x}  e_j-e_i[\fx{Tx}{y},e_ie_j,e_j].\notag
					\end{eqnarray}
					By the Five-Terms  identity \eqref{eq:five_term},
					we have
					\begin{eqnarray}\label{eq:ass term Cij final}
						&&[e_i,e_j,(e_ie_j)\fx{Tx}{y}]\\
						&=&e_i[e_j,e_ie_j,\fx{Tx}{y}]+[e_i,e_j,e_ie_j]\fx{Tx}{y}-[e_ie_j,e_ie_j,\fx{Tx}{y}]+[e_i,e_j(e_ie_j),\fx{Tx}{y}]\notag\\
						&=&e_i[e_j,e_ie_j,\fx{Tx}{y}]\notag\\
						&=&-e_i[\fx{Tx}{y},e_ie_j,e_j]\notag
					\end{eqnarray}
					Substituting   \eqref{eq:ass term Cij final}	 into \eqref{eq:Cij term 1}, we obtain
					\begin{equation}
						C_{ij}=e_i\fx{Tx}{y}e_i-e_j\fx{Ty}{x}  e_j+[e_i,e_j,(e_ie_j)\fx{Tx}{y}].
					\end{equation}
					
					We next compute $\sum A_{ij},	\sum_{i\neq j}(A_{ij}+B_{ij})$ and $	\sum_{i\neq j}(A_{ij}+C_{ij})$ for $x,y \in \re H$.
					We denote $$\alpha:=\fx{Tx}{y},\beta:=\fx{Ty}{x}$$ for simplicity.
					By identity \eqref{prop:real_part}, it is easy to obtain
					$$12 \re x-5x=\sum_{i=1}^7e_ix\overline{e_i}.$$
					It follows that 	\begin{eqnarray}\label{eq:sum A}
						\sum_{i, j=1}^7A_{ij}&=&98\fx{Tx}{y}+	\sum_{i, j=1}^7e_i\fx{Tx}{y}\overline{e_i}+\sum_{i, j=1}^7e_j\fx{Ty}{x}e_j\\
						&=&98\fx{Tx}{y}+7(12\re\fx{Tx}{y}-5\fx{Tx}{y})-7(12\re\fx{Ty}{x}-5\fx{Ty}{x} ) \notag\\
						&=&63\fx{Tx}{y}+35\fx{Ty}{x}+84(\re\fx{Tx}{y}-\re\fx{Ty}{x})
						\notag\\
						&=&63\alpha+35\beta+84(\re\alpha-\re\beta).
						\notag
					\end{eqnarray}
					
					Recall the identity (identity (3.4) in \cite{huo2024aacasubmod}):
					\begin{eqnarray}
						\mathit{Im}\, p&=&-\frac{1}{48}\sum_{i,j=1}^7[e_i,e_j,(e_ie_j)p]\label{eq:im}
					\end{eqnarray}
					for any $p\in \O$.
					It follows that
					\begin{eqnarray}\label{eq:A+B}
						\sum_{i,j=1,\atop i\neq j}^7(A_{ij}+B_{ij})&=&\sum_{i,j=1,\atop i\neq j}^72\fx{Tx}{y}-	\sum_{i,j=1,\atop i\neq j}^7[e_i,e_j,(e_ie_j)\fx{Ty}{x}]\\
						&=&84\fx{Tx}{y}-	\sum_{i, j=1}^7[e_i,e_j,(e_ie_j)\fx{Ty}{x}]\notag\\
						&=&84\fx{Tx}{y}+48\mathit{Im}\,  \fx{Ty}{x}\notag\\
						&=&84\alpha+48\mathit{Im}\,  \beta\notag
					\end{eqnarray}
					and \begin{eqnarray}\label{eq:A+C}
						\sum_{i,j=1,\atop i\neq j}^7(A_{ij}+C_{ij})&=&\sum_{i,j=1,\atop i\neq j}^72\fx{Tx}{y}+	\sum_{i,j=1,\atop i\neq j}^7[e_i,e_j,(e_ie_j)\fx{Tx}{y}]\\
						&=&84\fx{Tx}{y}+	\sum_{i, j=1}^7[e_i,e_j,(e_ie_j)\fx{Tx}{y}]\notag\\
						&=&84\fx{Tx}{y}-48\mathit{Im}\,  \fx{Tx}{y}\notag\\
						&=&84\alpha-48\mathit{Im}\,  \alpha\notag.
					\end{eqnarray}
					Hence
					\begin{eqnarray}
						&&	\sum_{i, j=1}^7A_{ij}-\dfrac{84}{48}\sum_{i,j=1,\atop i\neq j}^7(A_{ij}+B_{ij})-\dfrac{84}{48}	\sum_{i,j=1,\atop i\neq j}^7(A_{ij}+C_{ij})\\
						&=&63\alpha+35\beta+84(\re\alpha-\re\beta)-\dfrac{84}{48}\cdot 84\alpha-84\mathit{Im}\,  \beta-\dfrac{84}{48}\cdot 84\alpha+84\mathit{Im}\,  \alpha\notag\\
						&=&-147\alpha-49\beta.\notag
					\end{eqnarray}
					
					Combining this with
					\begin{eqnarray}
						m(x,y)=\fx{Tx}{y}+\fx{Ty}{x}=\alpha+\beta,
					\end{eqnarray}
					we have
					\begin{eqnarray}
						\sum_{i, j=1}^7A_{ij}-\dfrac{84}{48}\sum_{i,j=1,\atop i\neq j}^7(A_{ij}+B_{ij})-\dfrac{84}{48}	\sum_{i,j=1,\atop i\neq j}^7(A_{ij}+C_{ij})+49m(x,y)=-98\alpha.
					\end{eqnarray}
					Solving for $\alpha$ gives \eqref{eq:polarization-main}.	
				\end{proof}

\begin{eg}[Identity Operator]\label{ex:identity-verif}
For $T = I$ and $x,y \in \operatorname{Re} H$, let $\alpha = \langle x,y\rangle \in \mathbb{R}$. Then:
\begin{itemize}
\item $m(x,y) = 2\alpha$
\item $A_{ij} = 2\alpha$ (since $m(x,y\overline{e}_i) = 0$ by Lemma \ref{lem:inner_product})
\item $B_{ij} = C_{ij} = 0$ (by associator antisymmetry).
\end{itemize}
Substituting in \eqref{eq:polarization-main} {we obtain}:
\[
\langle Ix,y\rangle = \frac{1}{56}{\sum_{i,j=1,\atop i\neq j}^7} 4\alpha - \frac{1}{98}{\sum_{i,j=1}^7} 2\alpha - \alpha = \frac{168\alpha}{56} - \frac{98\alpha}{98} - \alpha = \alpha,
\]
confirming Theorem \ref{thm:polarization} for the identity operator.
\end{eg}

\begin{rem} \textit{The quaternionic case admits a significantly simpler polarization identity due to associativity \cite{semrl1986Hquadratic}. For a quaternionic Hilbert space (considered in the left-module setting) and operators acting on it, the polarization identity reduces to:}
\[
\langle Tx, y\rangle = \frac{1}{2} \left( m(x,y) + e_1 m(x,e_1 y) + m(x,e_2 y)e_2 + e_1 m(x,e_3 y)e_2 \right),
\]
\textit{where $\{1, e_1, e_2, e_3\}$ is a standard basis for the quaternions $\mathbb{H}$ with $e_1e_2 = e_3$. In contrast, the octonionic version (Theorem \ref{thm:polarization}) requires:}
\begin{itemize}
    \item \textit{Summation over all index pairs $(i,j)$,}
    \item \textit{Correction terms $B_{ij}$ and $C_{ij}$,}
    \item \textit{Additional coefficients accounting for non-associative deviations.}
\end{itemize}
\textit{This structural complexity reflects the fundamental role of non-associativity in octonionic functional analysis, necessitating more intricate formulas to recover sesquilinear forms from quadratic restrictions.}
				\end{rem}

\subsection{General Sesquilinear {Reconstruction}.}
 For arbitrary elements $x, y \in H$, we use the real decomposition
\[
x = \sum_{i=0}^{7} e_i x_i, \quad y = \sum_{j=0}^{7} e_j y_j \quad (x_i, y_j \in \operatorname{Re} H)
\]
and the component-wise polarization identity (Theorem \ref{thm:polarization}) to {reconstruct} the full sesquilinear form. Substituting into the expansion \eqref{eq:sesquilinear-decomposition} yields:

\begin{cor}[General Polarization Identity]\label{cor:general-polarization}
{Let $H$ be an Hilbert $\mathbb{O}$-bimodule and let $T$ be a right para-linear operator.}
The sesquilinear form {$\langle Tx,y\rangle$} is completely determined by the quadratic form on the slice cone, i.e.:
\begin{enumerate}
\item For general $x = \sum_{i=0}^7 e_i x_i$, $y = \sum_{j=0}^7 e_j y_j$ ($x_i,y_j \in \operatorname{Re} H$):
\begin{multline*}
\langle Tx, y \rangle = \\
\sum_{\ell,k=0}^{7} \overline{e_\ell} \left( \frac{1}{56} \sum_{\substack{i,j=1 \\ i \neq j}}^{7} \Big( 2A_{ij}(x_\ell, y_k) + B_{ij}(x_\ell, y_k) + C_{ij}(x_\ell, y_k) \Big) - \frac{1}{98} \sum_{i,j=1}^{7} A_{ij}(x_\ell, y_k) - \frac{1}{2} m(x_\ell, y_k) \right) e_k.
\end{multline*}

\item If $\langle Tx,x\rangle = 0$ for all $x \in \mathbb{C}(H)$, then $T = 0$
\end{enumerate}
\end{cor}

 This foundational result establishes $\mathbb{C}(H)$ as the definitive geometric domain for spectral analysis in octonionic Hilbert spaces. The restriction of the quadratic form to the slice cone completely determines para-linear operators, a property that underpins our characterization of self-adjointness (Theorem \ref{thm:slice_cone_char}) and spectral decomposition (Theorem \ref{thm:hilbert-schmidt}).

				\section{Para-Linear Self-Adjoint Operators}\label{sec:selfadjoint}	
Building on the algebraic and geometric foundations developed in the previous sections, this section turns to a central object in operator theory, namely para-linear self-adjoint operators. In classical Hilbert spaces, self-adjointness plays a pivotal role in spectral theory, but its extension to octonionic settings has long been obstructed by non-associativity. Within the framework of para-linearity and the slice cone geometry, we introduce and rigorously characterize self-adjoint para-linear operators. By using the octonionic polarization identity, we establish that self-adjointness corresponds precisely to real-valued quadratic forms on the slice cone. This geometric criterion forms the cornerstone for the spectral decomposition results that follows.

				  Throughout {this section}, let $H$ be a Hilbert $\mathbb{O}$-bimodule equipped with an $\mathbb{O}$-inner product $\langle \cdot, \cdot \rangle$.
				
				\begin{mydef}\label{def:selfadjoint}
					A para-linear operator $T: H \to H$ is \emph{self-adjoint} if $T = T^*$, where $T^*$ is the adjoint defined in \eqref{eq:adjoint_def}.
				\end{mydef}
				\begin{rem}
					The self-adjointness condition $T = T^*$ coincides with the real-linear adjoint since $T^* = T^{*_{\mathbb{R}}}$ when viewed in the real Hilbert space $(H, \langle \cdot, \cdot \rangle_{\mathbb{R}})$.
				\end{rem}
				
				We first characterize self-adjointness for $\mathbb{O}$-linear operators, where associativity simplifies the analysis:
				
				\begin{lemma}\label{lem:Olinear_equiv}
					Let $T: H \to H$ be {(right)} $\mathbb{O}$-linear, i.e., $T(xp) = (Tx)p$ for all $x \in H$ and $p \in \mathbb{O}$. The following are equivalent:
					\begin{enumerate}
						\item[(a)] $T$ is self-adjoint,
						\item[(b)] $\langle Tx, y \rangle = \langle x, Ty \rangle$ for all $x, y \in H$,
						\item[(c)] $\langle Tx, x \rangle \in \mathbb{R}$ for all $x \in H$.
					\end{enumerate}
				\end{lemma}
				
				\begin{proof}
					
					Since $T$ is $\O$-linear, it follows from  \eqref{eq:[yTx]} that
					 {$[y,T,x]=0$} for all $x,y\in H$. Thus $\langle Tx, y \rangle = \langle x, T^*y \rangle$. The equivalence (a)$\Leftrightarrow$(b) and (b)$\Rightarrow$(c) follow classically. For (c)$\Rightarrow$(b), see the sufficient part of the  proof of Theorem \ref{thm:slice_cone_char} below which is independent {from}  Lemma \ref{lem:Olinear_equiv}.
									\end{proof}
				
				For general para-linear operators, the characterization requires careful handling of non-associativity.
				
				\begin{lemma}\label{lem:para_linear_equiv}
					Let $T$ be para-linear and self-adjoint. The following are equivalent:
					\begin{enumerate}
						\item[(a)] $\langle Tx, y \rangle = \langle x, Ty \rangle$ for all $x, y \in H$;
						\item[(b)] $\langle Tx, x \rangle \in \mathbb{R}$ for all $x \in H$;
						\item[(c)] $T$ is $\mathbb{O}$-linear.
					\end{enumerate}
				\end{lemma}
				
				\begin{proof}
					(c)$\Rightarrow$(a) follows from Lemma \ref{lem:Olinear_equiv}, and (a)$\Rightarrow$(b) is immediate. We prove (b)$\Rightarrow$(c). Since $T$ is self-adjoint,
				 {	\[
					\langle Tx, x \rangle = \langle x, T^*x \rangle + [x, T, x] = \langle x, Tx \rangle + [x, T, x].
					\]}
					Condition (b) implies $\im (\langle Tx, x \rangle) = 0$, so 	{$[x, T, x] = 0$}. For arbitrary $x, y \in H$, the identity 	 {$[x + y, T, x + y] = 0$} yields
					 {	\[
					[ y, T,x] = -[ x, T,y] = -[ y, T^*,x] = -[ y, T,x],
					\]}
					where we used Lemma \ref{lem:associator_symmetry}. Thus 	 {$[ y, T,x] = 0$} for all $x, y \in H$. By  \eqref{eq:[yTx]}, $\langle B_{e_j}(T, x), y \rangle_{\mathbb{R}} = 0$ for all $j$ and $y$, so $B_{e_j}(T, x) = 0$. Hence $T$ is $\mathbb{O}$-linear.
				\end{proof}
				
				\begin{rem}\label{rem:nonlinear_exist}
	{A consequence of the previous result is that}	if a {para-linear,} self-adjoint $T$ is not $\mathbb{O}$-linear, there exists $x \in H$ such that $\langle Tx, x \rangle \notin \mathbb{R}$.
				\end{rem}
				
				The central result of this section provides a characterization of self-adjointness through the \emph{slice cone} $\mathbb{C}(H) := \bigcup_{J \in \mathbb{S}} (\operatorname{Re} H + J \operatorname{Re} H)$.
				
				\begin{thm}\label{thm:slice_cone_char}
					A para-linear operator $T: H \to H$ is self-adjoint if and only if $\langle Tx, x \rangle \in \mathbb{R}$ for all $x \in \mathbb{C}(H)$.
				\end{thm}
				
				To prove this  result, we  need to establish  the following lemma {at first}.
				\begin{lemma}		
					Let  $T $ be a   para-linear operator on $H$.	{Let us set}
$$Q(x):=\fx{Tx}{x},\qquad 2m(x,y):=Q(x+y)-Q(x-y).$$  Then for all $x,y\in \re H$ and all $J\in \S$, we have
					\begin{eqnarray}\label{eq:mJ}
						m(x,yJ)+m(y,xJ)=[m(x,y),J].
					\end{eqnarray}
				\end{lemma}
				\begin{proof}
					Note that $x,y\in \re H$. In view of Lemma \ref{lem:inner_product}, we have
					\begin{eqnarray*}
						m(x,yJ)+m(y,xJ)&=&\fx{Tx}{yJ}+\fx{T(yJ)}{x}+\fx{Ty}{xJ}+\fx{T(xJ)}{y}\\
						&=&(\fx{Tx}{y}+\fx{Ty}{x})J+\overline{J}(\fx{Ty}{x}+\fx{Tx}{y})\\
						&=&[m(x,y),J].			
					\end{eqnarray*}
				\end{proof}
				
				\begin{proof}[Proof of Theorem \ref{thm:slice_cone_char}]
					We  first prove the sufficiency.
					Assume that $T$  is self-adjoint.
					Then for all $x\in \re H$, we have
					$$\fx{Tx}{x}=\fx{x}{T^*x}=\fx{x}{Tx}=\overline{\fx{Tx}{x}}$$ and hence $\fx{Tx}{x}\in \R$.
					For any $z=x+yJ\in \mathbb{C} (H)$ with $x,\,y\in\re H$ and $J\in \mathbb{S}$, we conclude from Lemma \ref{lem:inner_product} that
					\begin{eqnarray*}
						\fx{Tz}{z}&=&\fx{Tx}{x}+\fx{T(yJ)}{x}+\fx{Tx}{yJ}+\fx{T(yJ)}{yJ}\\
										&=&\fx{Tx}{x}+\abs{J}^2\fx{Ty}{y}+2\re(\fx{Tx}{y}J)\in \R.
					\end{eqnarray*}

					Next we  prove the necessity.  Fix $x$, $y\in \re H$ arbitrarily. 	Let 	$$A_{ij}=m(x,y)+e_im(x,y\overline{e_i})+m(x,y\overline{e_j})e_j,$$ and $$B_{ij}=[e_im(x,y\overline{(e_ie_j)})]e_j,\quad C_{ij}=e_i[m(x,y\overline{(e_ie_j)})e_j]$$ as usual.
					Since  $\fx{Tx}{x}\in \R \text{ for all }x\in\mathbb{C} (H)$, it follows that
					$$2m(x,ye_i):=Q(x+ye_i)-Q(x-ye_i)\in \R$$ for $i=0,\dots,7$. Hence
					\begin{eqnarray}\label{eq:BC}
						B_{ij}(x,y)=C_{ij}(x,y).
					\end{eqnarray}

					Noticing that $x,y\in \re H$,  it follows  that $	\fx{x}{T^*y}=\fx{Tx}{y}.$
					In view of  Theorem \ref{thm:polarization}, we conclude from \eqref{eq:BC} that
					\begin{eqnarray}\label{eq:<xT*y>}
					\qquad	\fx{x}{T^*y}=\fx{Tx}{y}=\dfrac{1}{56}	\sum_{i,j=1,\atop i\neq j}^7(2A_{ij}+2B_{ij})(x,y)-\dfrac{1}{98}	\sum_{i, j=1}^7A_{ij}(x,y)-\dfrac{1}{2}m(x,y).
					\end{eqnarray}
					We now show that $$\fx{x}{T^*y}=\fx{x}{Ty}.$$
					To this end, we compute $	\overline{A_{ij}(y,x)}$. Since $m(x,ye_i)\in \R$ for each $i=0,\dots,7$, it follows from  \eqref{eq:mJ} that
					\begin{eqnarray*}
						\overline{A_{ij}(y,x)}&=&m(y,x)-e_im(y,x\overline{e_i})-e_jm(y,x\overline{e_j})\\
						&=&m(y,x)+e_im(x,y\overline{e_i})+e_jm(x,y\overline{e_j}).
					\end{eqnarray*}
					Noticing that $$m(x,y)=m(y,x),$$we conclude  that
					\begin{eqnarray}\label{eq:Aij=bar}
						\overline{A_{ij}(y,x)}=m(x,y)+e_im(x,y\overline{e_i})+e_jm(x,y\overline{e_j})=A_{ij}(x,y).
					\end{eqnarray}
					Similarly, we also have
					\begin{eqnarray}\label{eq:Bij=bar}
						\overline{B_{ij}(y,x)}=B_{ij}(x,y).
					\end{eqnarray}
					Utilizing \eqref{eq:Aij=bar} and \eqref{eq:Bij=bar}, we conclude  from  Theorem \ref{thm:polarization} again   that
					\begin{eqnarray*}
						\fx{x}{Ty}&=&\overline{\fx{Ty}{x}}\\
						&=&\dfrac{1}{56}	\sum_{i,j=1,\atop i\neq j}^72\overline{A_{ij}(y,x)}+2\overline{B_{ij}(y,x)}-\dfrac{1}{98}	\sum_{i, j=1}^7\overline{A_{ij}(y,x)}-\dfrac{1}{2}\overline{m(y,x)}\\
						&=&\dfrac{1}{56}	\sum_{i,j=1,\atop i\neq j}^7(2A_{ij}+2B_{ij})(x,y)-\dfrac{1}{98}	\sum_{i, j=1}^7A_{ij}(x,y)-\dfrac{1}{2}m(x,y).		
					\end{eqnarray*}
					Combing this with \eqref{eq:<xT*y>}, we obtain that
					$${\fx{x}{T^*y}}=\fx{x}{Ty}$$  for arbitrarily fixed $x,\, y\in \re H$.
					This yields $Ty=T^*y$ for arbitrarily fixed $y\in \re H$.
					Note that $T^*$ is also para-linear by Theorem \ref{thm:T* is paralinear}.  Thus we get from   Lemma \ref{cor: A_p(x,f)=0 and f(px)=pf(x)} that $T=T^*$.
				\end{proof}

				\begin{rem}
					Theorem \ref{thm:slice_cone_char} holds analogously for quaternionic Hilbert bimodules, emphasizing the universality of the slice cone characterization.
				\end{rem}

\section{Strong Eigenvalues and Eigenvectors}

The  non-associativity of octonions {entails} a refined approach to spectral theory, particularly regarding the eigenvalue problem. The classical eigenvalue equation \( T z = \lambda z \) fails to maintain coherence within the framework of para-linear operators, since the left multiplication by \( \lambda \) does not align with right para-linearity. Thus, we focus on slice paravectors, elements contained entirely within a single complex slice, and introduce the spectral parameter acting on the right. This naturally motivates the concept of a \textit{strong eigenpair}.

{In this section,} \(H\) denotes a Hilbert \(\mathbb O\)-bimodule, {while $\mathbb C(H)$ denotes the slice cone}.

\subsection{Definition and Fundamental Properties}

\begin{mydef}[Strong eigenpair]
Let \(T \in \mathscr{B}_{\mathcal{R}\mathcal{O}}(H) \). A nonzero slice paravector \(  z \in \mathbb C(H) \) is called a \textbf{strong eigenvector} associated to a \textbf{strong eigenvalue} \( \lambda \in \mathbb O \) if
\[
T z = z \lambda.
\]
\end{mydef}

\begin{rem}
The choice \( T z = z\lambda \) is consistent with right para-linearity, preserving slice coherence. Restricting \( z \) to \( \mathbb C(H) \) ensures compatibility with the octonionic geometric structure, thus establishing a meaningful spectral theory.

In this paper we have been considering the notion of right para-linearity. Although the notions of left and right para-linearity are different, but somewhat equivalent from a mathematical point of view, for physical applications the side on which linearity (and so para-linearity) is considered is crucial. This is relevant in quaternionic quantum mechanics, see the well-known Adler's book \cite{MR1333599}, in which it is shown that right linearity and right eigenvalues are the relevant concepts while the left versions seem less useful. This is the ground for our choice of the side of para-linearity.

{
We note that  quaternionic spectral theory is well established since the definition of the S-spectrum in 2006 (see the Introduction of  \cite{MR3887616} for some historical remarks). Since then, there has been an extensive development of the theory, including new functional calculi and spectral decompositions tailored specifically to quaternionic linear operators, see e.g.  \cite{Colombo2008funcalculus,Colombo2011ADVqevol,MR3587903,Alpay2016JMP,Colombo2019normal},
the books \cite{colombo2011noncomfunctcalculus,MR3887616,MR3967697} and the references therein.
This theory makes use of the appropriate notion of spectrum for quaternionic operators and allows an associative-based spectral decomposition approach.}
\end{rem}

\begin{eg}
Consider the octonionic Hermitian matrix
\[
M = \begin{pmatrix}
t & p \\
 {\overline{p} }& s
\end{pmatrix},\quad t, s \in \mathbb{R},\, p \in\mathbb  O \setminus \{0\}.
\]
Selecting \( J \in \mathbb S \) such that \( p \in \mathbb{C}_J \), the restriction of \( M \) to \( \mathbb{C}_J^2 \subset \mathbb C(H) \) with $H=\mathbb O^2$ yields a classical complex Hermitian operator, whose eigenvalues are strong eigenvalues, each admitting strong eigenvectors within \( \mathbb{C}_J^2 \).
\end{eg}

\subsection{Reality of Strong Eigenvalues for Self-adjoint Operators}

{Theorem 8.6 shows that self-adjoint operators are characterized by the fact that the quadratic form \( \langle T x, x \rangle \) is real-valued, for all \( x \in \mathbb C(H) \). The following result is an immediate consequence:}

\begin{cor}\label{StrongEigen}
Let {\(T \in \mathscr{B}_{\mathcal{R}\mathcal{O}}(H) \) and} \( T = T^* \) and suppose \( T z = z \lambda \), where \( z \neq 0 \) is a strong eigenvector. Then \( \lambda \in \mathbb{R} \).
\end{cor}

\begin{proof}
By Theorem \ref{thm:slice_cone_char}, $\langle Tz, z\rangle \in \mathbb{R}$. Using self-adjointness and Lemma 	\ref{lem:inner_product}(3), we get
\[
\langle Tz, z\rangle = \langle z\lambda, z\rangle = \overline{\lambda} \langle z, z\rangle.
\]
Since $\langle z, z\rangle > 0$, it follows that $\overline{\lambda} \in \mathbb{R}$, so $\lambda \in \mathbb{R}$.
\end{proof}

\subsection{Orthogonality and Weak Associative Bases}

The properties of para-linear projections necessitate a refined notion of orthogonality to ensure that spectral decompositions behave as in the classical case within slices.

{We begin with the notion of weak associative  orthogonality  \cite{huoqinghai2021tensor}.}
\begin{mydef}
A family \( \{x_\alpha\}_{\alpha \in \Lambda} \subset H \) is said to be:
\begin{itemize}
    \item \textbf{orthonormal} if \( \langle x_\alpha, x_\beta \rangle = \delta_{\alpha\beta} \);
    \item \textbf{weak associative orthonormal} if it is orthonormal and satisfies \( B_p(x_\alpha,x_\beta)=0 \) for all \( p \in\mathbb  O \);
    \item a \textbf{weak associative orthonormal basis} if it is maximal with respect to these properties.
\end{itemize}
\end{mydef}

This formulation ensures the validity of Parseval's identity,  {as in the} classical Hilbert space theory.

\begin{thm}[Parseval    Identity {\cite[Thm.~5.7]{huoqinghai2021tensor}}]\label{thm:parseval}
Let \(\{x_\alpha\}\) be a weak associative orthonormal basis of \(H\). For every \( x \in H \),
\[
x = \sum_{\alpha} x_\alpha \langle x_\alpha, x \rangle, \quad \|x\|^2 = \sum_{\alpha} |\langle x_\alpha, x \rangle|^2.
\]
\end{thm}

Moreover, eigenvectors corresponding to distinct eigenvalues are orthogonal and weakly associative, {in fact we have the following:}

\begin{thm}[Orthogonality of distinct eigenvectors]\label{thm:eigen-orthogonality}
Suppose {\(T \in \mathscr{B}_{\mathcal{R}\mathcal{O}}(H) \),} \(T = T^*\) and \( T z_k = z_k \lambda_k \) for \( z_k \in \mathbb C(H) \setminus \{0\} \), with distinct eigenvalues \( \lambda_1 \neq \lambda_2 \). Then:
\begin{enumerate}
    \item 	 {\([z_2,T,z_1] = 0\);}
    \item \(\langle z_1,z_2\rangle = 0\);
    \item \(\{z_1,z_2\}\) is weak associative orthonormal.
\end{enumerate}
\end{thm}
\begin{proof}
					Write $z_k = x_k + J_k y_k$ for $x_k, y_k \in \operatorname{Re} H$ and $J_k \in \mathbb{S}$ ($k=1,2$). By Corollary \ref{StrongEigen}, {we know that} $\lambda_k \in \mathbb{R}$.
					
					\noindent{Let us prove point (1).} From  \eqref{eq:adjoint_def} and self-adjointness of $T$ we have
					\begin{equation}\label{eq:eigen-diff}
						(\lambda_1 - \lambda_2) \langle z_1, z_2 \rangle = \langle Tz_1, z_2 \rangle - \langle z_1, T z_2 \rangle = -{[ z_2, T,z_1].}
					\end{equation}
					By Lemma \ref{lem:vanishing_associator} and \eqref{eq:[yTx]} we deduce
						\begin{equation}\label{eq:associator-decomp}
						{[ z_2, T, z_1]} = \sum_{i=1}^7 e_i \langle z_2, B_{e_i}(T, J_1 y_1) \rangle_{\mathbb{R}} = \sum_{i=1}^7 e_i \operatorname{Re} \bigl( \langle T y_1, y_2 \rangle [e_i, J_2, J_1] \bigr).
					\end{equation}
					\begin{itemize}
						\item \textit{Case $J_1 = \pm J_2$}: The associator $[e_i, J_2, J_1]$ vanishes by antisymmetry, so {the right-hand side is $0$}.
						\item \textit{Case $J_1 \neq \pm J_2$}:
Define \(J_{2}' := J_{2} - c\,J_{1}\) with \(c = \langle J_{2}, J_{1} \rangle_{\mathbb{R}}\). Then \(J_{1} \perp_{\mathbb{R}} J_{2}'\),  {where \(\perp_{\mathbb{R}}\) denotes orthogonality in the Euclidean space \(\mathbb{O} \cong \mathbb{R}^{8}\). In particular,}
\[
[e_{i}, J_{2}, J_{1}] = [e_{i}, J_{2}', J_{1}] \qquad (i=0,\dots,7).
\]
{By a suitable scaling and change of basis, we may assume without loss of generality that
\[
J_{1}=e_{1}, \qquad J_{2}'=e_{2}.
\]
Let \(\mathbb{H}_{e_{1},e_{2}}\) be the quaternionic subalgebra generated by \(e_{1}\) and \(e_{2}\), and let \(\mathbb{H}_{e_{1},e_{2}}^{\perp_{\mathbb{R}}}\) denote its Euclidean orthogonal complement in \(\mathbb{O}\). A direct computation shows that, for any \(z_{1},z_{2}\in H\),
\[
[z_{2},T,z_{1}]
= \sum_{i=1}^{7} e_{i}\,\operatorname{Re}\!\big(\langle T y_{1}, y_{2}\rangle\,[e_{i},e_{2},e_{1}]\big)
\;\in\; \mathbb{H}_{e_{1},e_{2}}^{\perp_{\mathbb{R}}},
\]
where \([\cdot,\cdot,\cdot]\) is the  associator on $\O$. On the other hand,  {by direct calculations we deduce}
\[
(\lambda_{1}-\lambda_{2})\,\langle z_{1}, z_{2}\rangle \in \mathbb{H}_{e_{1},e_{2}},
\]
and hence \eqref{eq:eigen-diff} forces
\[
[z_{2},T,z_{1}] \in \mathbb{H}_{e_{1},e_{2}} \cap \mathbb{H}_{e_{1},e_{2}}^{\perp_{\mathbb{R}}}=\{0\}.
\]
Therefore \( [z_{2},T,z_{1}] = 0\).
 }
					\end{itemize}

\medskip

\noindent{Let us prove point (2).} Point (1) and \eqref{eq:eigen-diff} imply $\langle z_1, z_2 \rangle = 0$.

\medskip
					
					\noindent{Let us prove point (3).} To show $B_p(z_1, z_2) = 0$ for all $p \in \mathbb{O}$:
					\[
					B_p(z_1, z_2) = \langle J_1 y_1, J_2 y_2 \rangle p - \langle J_1 y_1, (J_2 y_2) p \rangle = [\overline{J_1}, J_2, p] \langle y_1, y_2 \rangle.
					\]
					\begin{itemize}
						\item \textit{Case $J_1 = \pm J_2$}: The associator vanishes identically.
						\item \textit{Case $J_1 \neq \pm J_2$}: Expand $\langle z_1, z_2 \rangle = 0$:
						\[
						\langle x_1, x_2 \rangle + \overline{J_1} \langle y_1, x_2 \rangle + \langle x_1, y_2 \rangle J_2 + \overline{J_1} J_2 \langle y_1, y_2 \rangle = 0.
						\]
						Since $\{1, J_1, J_2, J_1 J_2\}$ is linearly independent in the real vector space $\O$, it follows that $\langle y_1, y_2 \rangle = 0$. Thus $B_p(z_1, z_2) = 0$.
					\end{itemize}
					This completes the proof.
				\end{proof}

The framework of strong eigenpairs isolates the slice-compatible part of the spectrum and underpins a rigorous spectral decomposition. The following key facts will be used (note that items~(1)–(2) have already been established):
\begin{enumerate}
  \item Strong eigenvalues of self-adjoint para-linear operators are real.
  \item Strong eigenvectors associated with distinct eigenvalues form orthogonal, weakly associative systems.
  \item For \(z\in C(H)\), the slice projections \(P_{z}:x\mapsto z\langle z,x\rangle\) are para-linear, self-adjoint, and idempotent, enabling spectral diagonalization (see Section~\ref{sec:projections}).
\end{enumerate}

			\section{Projection Operators in Octonionic Hilbert Spaces}
\label{sec:projections}

This section develops the theory of projection operators associated with slice paravectors in octonionic Hilbert spaces. These operators are fundamental for the spectral decomposition of para-linear self-adjoint operators {that will be} established in   Section \ref{sec:spectral}.
\subsection{Slice Projection Operators}

Let \(H\) be a Hilbert \(\mathbb{O}\)-bimodule. For any \textit{slice paravector} \(z \in \mathbb{C}(H) \), we define the associated \textit{slice projection operator} \(P_z : H \to H\) by
\[
P_z(x) := z \langle z, x \rangle, \quad x \in H.
\]
The requirement \(z \in \mathbb{C}(H)\) is essential for the following {fundamental} properties.

\begin{thm}[Properties of Slice Projections]
\label{thm:projection-properties}
For any \(z \in \mathbb{C}(H)\) with \(\|z\| = 1\), the operator \(P_z\) satisfies:
\begin{enumerate}
    \item \textit{Para-linearity}: \(P_z \in \mathscr{B}_{\mathcal{RO}}(H)\),
    \item \textit{Self-adjointness}: \(P_z^* = P_z\),
    \item \textit{Idempotence}: \(P_z \circledcirc P_z = P_z\),
    \item \textit{Range and kernel}: \(\ker P_z = z^{\perp}\) and \(\operatorname{Ran} P_z =  z\mathbb{O}\).
\end{enumerate}
\end{thm}

\begin{proof}
Write \(z = u + Jv\) with \(u, v \in \operatorname{Re} H\) and \(J \in \mathbb{S}\).\\

\noindent\textit{(1) Para-linearity}: For \(p \in \mathbb{O}\) and \(x \in H\), consider the second associator
\[
B_p(P_z, x) = P_z(x)p - P_z(xp) = (z\langle z, x\rangle) p - z\langle z, xp\rangle.
\]
Using the real part operator and Lemma~	\ref{lem:inner_product}(3), {we deduce}
\[
\operatorname{Re} B_p(P_z, x) = \operatorname{Re}\left( [z, \langle z, x\rangle, p] + z B_p(z, x) \right)=\operatorname{Re}(z B_p(z, x)).
\]
For $z = u + Jv \in \mathbb{C}(H)$ with $u, v \in \operatorname{Re} H$, we compute from \eqref{eq:Bp(uv)=-Bp(vu)} and Lemma \ref{lem:inner_product}(4) that:
\begin{align*}
	B_p(z, x) &= B_p(u + Jv, x) \\
	&= B_p(Jv, x) \\
	&= -B_p(x, Jv). 
\end{align*}
Then:
\begin{align*}
	\operatorname{Re}(z B_p(z, x)) &= -\operatorname{Re}((u + Jv) B_p(x, Jv)) \\
	&= -\operatorname{Re}(u B_p(x, Jv)) - \operatorname{Re}(Jv B_p(x, Jv)) \\
	&=-v\operatorname{Re}(J B_p(x, Jv)).
\end{align*}
Now, using identity \eqref{eq:B1} we get:
$$
	-v\operatorname{Re}(JB_p(x, Jv)) 	= -v \operatorname{Re}(J B_p(\langle x, \cdot \rangle, vJ)) = 0
$$
This completes the proof that $\operatorname{Re} B_p(P_z, x) = 0$.
Thus \(P_z\) is para-linear.\\

\noindent\textit{(2) Self-adjointness}: By Theorem~\ref{thm:T* is paralinear}, \(P_z^*\) is para-linear. For \(x, y \in H\),
{{it follows from Lemma~ 	\ref{lem:inner_product} that} \begin{align*}
		\langle P_z(x), y \rangle_{\mathbb{R}} = \operatorname{Re} \left( \overline{\langle z, x \rangle} \langle z, y \rangle \right) = \operatorname{Re} \left( \overline{\langle z, y \rangle } {\langle z, x \rangle}\right) = \langle  P_z(y),x \rangle_{\mathbb{R}}.
\end{align*}}
  {In view of Remark \ref{rem:T*=T*R}, we have  \(P_z^* = P_z\).}\\

\noindent\textit{(3) Idempotence}: By \eqref{eq:Bp(uv)=-Bp(vu)} and \(\|z\| = 1\):
\[
P_z \circledcirc P_z (x) = z \langle z, z \langle z, x \rangle \rangle = z \left( \langle z, z \rangle \langle z, x \rangle \right) = z \langle z, x \rangle = P_z(x).
\]

\noindent\textit{(4) Range and kernel}: Immediate from the definition and \(\langle z, z \rangle = 1\).
\end{proof}

\begin{rem}
\label{rem:necessity-slice-condition}
The slice condition \(z \in \mathbb{C}(H)\) is necessary for para-linearity. For counterexamples such as \(z = (e_1, e_2) \in \mathbb{O}^2 \setminus \mathbb{C}(\mathbb{O}^2)\), the map \(x \mapsto z \langle z, x \rangle\) fails to be para-linear.
\end{rem}

\subsection{Interplay with Spectral Theory}

Projection operators commute with para-linear operators when restricted to eigenspaces, {as proved in the next result}.

\begin{thm}[Eigenprojection Commutation]
\label{thm:projection-commutation}
Let \(T \in \mathscr{B}_{\mathcal{RO}}(H)\) be para-linear and self-adjoint, and \(z \in \mathbb{C}(H)\) a strong eigenvector with real eigenvalue \(\lambda\) (Corollary~\ref{StrongEigen}). Then
\[
T \circledcirc  P_z = P_z \odot \lambda.
\]
\end{thm}

\begin{proof}
For \(x \in \operatorname{Re} H\):
\[
(T \circledcirc  P_z)(x) = T(P_z(x)) = T(z \langle z, x \rangle).
\]
 Noting that $\lambda\in \R$ and \(Tz = z\lambda\), it follows  that
 $$T(z \langle z, x \rangle) = (Tz) \langle z, x \rangle - B_{\langle z, x \rangle}(T, z) = \lambda z \langle z, x \rangle - B_{\langle z, x \rangle}(T, z).$$
Since   \(z = u + Jv\) with \(u, v \in \operatorname{Re} H\), by part (3) of Theorem \ref{thm:para_linear_char} and part (1) of Theorem \ref{thm:tensor_decomp}, we have
\begin{eqnarray*}
	B_{\fx{z}{x}}(T,z)	&=&\sum_{i=1}^7
	e_i \re T([z,\fx{z}{x},e_i])\\
	&=&\sum_{i=1}^7
	e_i \re T([a+Jb,\fx{a+Jb}{x},e_i])\\
	&=&\sum_{i=1}^7\fx{b}{x}
	e_i \re T(b[J,\overline{J},e_i])\\
	&=&0.
\end{eqnarray*}
 The result follows by para-linearity and agreement on \(\operatorname{Re} H\).
\end{proof}

The commutation relation {proved above} is pivotal for the spectral decomposition in Section~\ref{sec:spectral}, where self-adjoint operators are diagonalized via slice projections.

 \section{Spectral Decomposition of Compact Self-Adjoint Operators}
\label{sec:spectral}

This section establishes the spectral decomposition for compact self-adjoint para-linear operators on octonionic Hilbert spaces, resolving {another} fundamental challenge in non-associative functional analysis. The absence of associativity necessitates a geometrically constrained spectral theory, culminating in a Hilbert-Schmidt-type theorem that generalizes classical results, while respecting the octonionic structure.

\subsection{Strong Eigenvalues and Standard Spectral Data}

Let $H$ be a Hilbert $\mathbb{O}$-bimodule and $T \in \mathscr{B}_{\mathcal{RO}}(H)$ be a \textit{compact self-adjoint para-linear operator}.
For simplicity, we confine ourselves to the point spectrum case. The spectrum $\sigma(T)$ consists of $\lambda \in \mathbb{O}$ where the eigenspace $$E_{\lambda}(T) := \{x \in H : Tx = x\lambda\}$$ is non-trivial. To ensure spectral significance, we restrict to  strong eigenvalues compatible with slice geometry.

By Corollary \ref{StrongEigen},  strong eigenvalues are real. For {the} spectral decomposition, we require:

\begin{mydef}[Standard Strong Eigenvalue]\label{def:ssEngen}
	A strong  eigenvalue $\lambda \in \sigma(T)$ is called  a \textbf{standard strong eigenvalue} if there exists an orthonormal {system} $\{z_i\}_{i \in \Lambda} \subset \mathbb{C}(H)$ of strong eigenvectors such that
	\begin{enumerate}
		\item $Tz_i = z_i\lambda$ for all $i \in \Lambda$;
		\item $ E_\lambda(T) \subset {\operatorname{span}}_{\mathbb O}\{z_i\}_{i\in\Lambda}:= \left\{ \sum\limits_{k=1}^N z_{i_k} p_k : N \in \mathbb{N},\  i_k \in \Lambda,\  p_k \in \mathbb{O} \right\}$.
	\end{enumerate}
\end{mydef}

This ensures {that} eigenspaces are generated by slice-orthogonal systems in $\mathbb{C}(H)$.

\subsection{Main Spectral Theorem}

The following represents {a fundamental result in} octonionic spectral theory:

\begin{thm}[Hilbert-Schmidt Theorem]
\label{thm:hilbert-schmidt}
Let $T \in \mathscr{B}_{\mathcal{RO}}(H)$ be compact, self-adjoint, and para-linear with all eigenvalues standard strong. Then there exists  a weak associative orthonormal basis $\{z_i\}_{i \in \mathcal{N}} \cup \{z_{\alpha}\}_{\alpha \in \Lambda}$ such that:
\begin{enumerate}
    \item $Tz_i = \lambda_i z_i$ with $\lambda_i \in \mathbb{R} \setminus \{0\}$ for $i \in \mathcal{N}$, where $\mathcal{N}$ is finite or countable. If $\mathcal{N}$ infinite, $\lim_{i\to\infty} \lambda_i = 0$;
    \item $Tz_{\alpha} = 0$ for $\alpha \in \Lambda$;
    \item For any $  x \in H$,
    \[
    Tx = \sum_{i \in \mathcal{N}} \lambda_i P_{z_i}x
    \]
    with norm convergence, where $P_{z_i}(x) = z_i \langle z_i, x \rangle$ is the para-linear projection.
\end{enumerate}
\end{thm}

\begin{proof}
    Since each eigenvalue of $T$ is a standard strong eigenvalue and $T$ is an $\mathbb{R}$-linear compact operator, its real spectrum $\sigma_{\mathbb{R}}(T)$ coincides with $\sigma(T)$. By classical theory for real linear compact operators, $\sigma(T)$ is at most countable with $0$ as the only possible accumulation point. Thus, $\mathcal{N}$ is at most countable and $\lim_{i\to\infty} \lambda_i = 0$.

    By Definition \ref{def:ssEngen}, for each $\lambda \in \sigma(T)$, there exists an orthonormal {system} $\{z_i\}_{i\in\Pi}$ of strong eigenvectors such that
    \[
        E_{\lambda}(T) \subset   {\operatorname{span}}_{\mathbb O}\{z_i\}_{i\in\Lambda}.
    \]
    For any $z_i, z_j$ with $i \neq j \in \Lambda$, if there exists $J \in \mathbb{S}$ such that $z_i, z_j \in \mathbb{C}_J(H)$, then
    \[
        B_p(z_i, z_j) = [p, \bar{J}, J] \fx{y_i}{y_j} = 0,
    \]
    where $y_i = \re(z_i\bar{J})$ and $y_j$ is defined similarly.

    Now suppose $z_i \in \mathbb{C}_I(H)$, $z_j \in \mathbb{C}_J(H)$ with $I \neq \pm J$. From $\fx{z_i}{z_j} = 0$ and Theorem \ref{thm:eigen-orthogonality}(3), we obtain $B_p(z_i, z_j) = 0$. By Theorem \ref{thm:eigen-orthogonality}, the set $\{z_i\}_{i\in\mathcal{N}} \cup \{z_\alpha\}_{\alpha\in\Lambda}$ forms a weak associative orthonormal set.

    Viewing $T$ as a real linear compact operator, we have
    \[
        H = \bigoplus_{i\in\mathcal{N}} E_{\lambda_i}(T) \oplus \ker T \subset  {\operatorname{span}}_{\O}\left( \{z_i\}_{i\in\mathcal{N}} \cup \{z_\alpha\}_{\alpha\in\Lambda} \right) \subset  H.
    \]
    Thus, $\{z_i\}_{i\in\mathcal{N}} \cup \{z_\alpha\}_{\alpha\in\Lambda}$ is a weak associative orthonormal basis for $H$.

    Next, we show that $Tx = \sum_{i\in\mathcal{N}} \lambda_i P_{z_i}(x)$ for all $x \in H$. Since $\lim_{i\to\infty} \lambda_i = 0$, for any $\epsilon > 0$, there exists $N \in \mathbb{N}$ such that $|\lambda_i| \leqslant  \epsilon$ for $i > N$. By the Parseval Theorem (Theorem \ref{thm:parseval}),
    \begin{align*}
        \left\| \sum_{i\in\mathcal{N}} \lambda_i P_{z_i}(x) \right\|^2
        &= \left\| \sum_{i\in\mathcal{N}} \lambda_i z_i \fx{z_i}{x} \right\|^2 \\
        &= \sum_{i=1}^N |\lambda_i \fx{z_i}{x}|^2 + \sum_{i=N+1}^\infty |\lambda_i \fx{z_i}{x}|^2 \\
        &\leqslant \sum_{i=1}^N |\lambda_i \fx{z_i}{x}|^2 + \epsilon^2 \|x\|^2.
    \end{align*}
    This proves norm convergence, so the operator $Px = \sum_{i\in\mathcal{N}} \lambda_i P_{z_i}(x)$ is para-linear. To show $T = P$, by Lemma \ref{cor: A_p(x,f)=0 and f(px)=pf(x)}(2), it suffices to verify $Tx = Px$ for $x \in \operatorname{Re} H$. The Parseval Theorem gives
    \[
        x = \sum_{i\in\mathcal{N}} P_{z_i}(x) + \sum_{\alpha\in\Lambda} P_{z_\alpha}(x).
    \]
    Since $x \in \operatorname{Re} H$ and $T$ is para-linear and continuous, Lemma \ref{thm:projection-commutation} yields
    \begin{align*}
        Tx &= T \left( \sum_{i\in\mathcal{N}} P_{z_i}(x) + \sum_{\alpha\in\Lambda} P_{z_\alpha}(x) \right) \\
           &= \sum_{i\in\mathcal{N}} (T \circledcirc P_{z_i})(x) + \sum_{\alpha\in\Lambda} (T \circledcirc P_{z_\alpha})(x) \\
           &= \sum_{i\in\mathcal{N}} (P_{z_i} \odot \lambda_i)(x) \\
           &= \sum_{i\in\mathcal{N}} \lambda_i P_{z_i}(x).
    \end{align*}
    The last equality holds because $T \circledcirc P_{z_\alpha} = 0$ for $z_\alpha \in \ker T$. This completes the proof.
\end{proof}
\begin{rem}
	If $T$ can be decomposed into a sum of projections
	$$T=\sum_{i}\lambda_iP_{z_i}$$ with $\fx{z_i}{z_j}=\delta_{ij}$, then by the requirement of para-linearity, we must have $z_i\in \spc(H)$ and   $$Tz_i=\lambda_iz_i.$$
	This implies each $\lambda_i$ is a standard eigenvalue of $T$. That is, to decompose a para-linear operator $T$ into a sum of projections, it is natural to require the eigenvalues to be  standard strong.
	
\end{rem}

\subsection{Finite Rank Operators}

The finite rank case admits a particularly elegant decomposition.
\begin{mydef}[Finite Rank Operator]
$T \in \mathscr{B}_{\mathcal{RO}}(H)$ has \textbf{rank} $n$ if $\operatorname{Ran} T$ is an $\O$-submodule and  $$\dim_{\mathbb{O}}(\operatorname{Ran} T)  := \dim_{\mathbb{R}} \operatorname{Re}(\operatorname{Ran} T)= n.$$
\end{mydef}

\begin{thm}[Finite Rank Spectral Decomposition]
\label{thm:finite-rank}
Let {the  self-adjoint, para-linear} $T$ have rank $n$, with $n$ pairwise orthogonal strong eigenvectors $\{z_k\}_{k=1}^n \subset \mathbb{C}(H)$ and eigenvalues $\lambda_k \in \mathbb{R}$. Then
\[
T = \sum_{k=1}^n \lambda_k P_{z_k}.
\]
\end{thm}

\begin{proof}
By Theorem \ref{thm:hilbert-schmidt}, $\{z_k\}$ is weak associative orthonormal. Let $M = \operatorname{Ran} T$. If there exists $  w \in M$ with $\|w\|=1$ and $\langle w, z_k \rangle = 0$ for any $k$, then
\[
\mathcal{S} = \{ w \} \cup \{ z_k e_i : 1 \leqslant k \leqslant n,  0 \leqslant i \leqslant 7 \}
\]
is orthonormal in $(M, \langle \cdot, \cdot \rangle_{\mathbb{R}})$. But $|\mathcal{S}| = 1 + 8n > \dim_{\mathbb{R}} M = 8n$, contradiction. Thus $\{z_k\}$ $\mathbb{O}$-spans $M$, and decomposition follows from Theorem \ref{thm:hilbert-schmidt}.
\end{proof}

\begin{cor}
If $T$ {has rank $n$, is self-adjoint} and para-linear, and has $n$ distinct strong eigenvalues, then $T = \sum_{k=1}^n \lambda_k P_{z_k}$.
\end{cor}

\subsection{Geometric Interpretation}

The spectral decomposition reveals how octonionic non-associativity shapes diagonalization:
\begin{itemize}
    \item \textbf{Slice-Constrained Eigenvectors}: Eigenvectors in $\mathbb{C}(H)$ ensure para-linearity of projections $P_z$ (Theorem \ref{thm:projection-properties}).
    \item \textbf{Weak Associativity}: Weak associative orthogonality $B_p(z_i, z_j) = 0$ enables convergent projection sums and
    \begin{equation} \label{eq:projection_orthogonality}
    	P_{z_j} P_{z_i} = \delta_{ij} P_{z_i} \quad \forall\ i,j.
    \end{equation}
    \item \textbf{Real Spectrum}: Strong eigenvalues $\lambda_i \in \mathbb{R}$ reflect self-adjointness (Corollary \ref{StrongEigen}).
\end{itemize}
This framework transforms non-associativity from an obstruction into a structured feature, enabling rigorous spectral analysis in octonionic Hilbert spaces.

\section{Functional Calculus for Compact Self-Adjoint Operators}
\label{sec:functional_calculus}

This section develops a functional calculus for compact, self-adjoint, para-linear operators with standard strong eigenvalues on octonionic Hilbert spaces.
{The non-associative nature of $\mathbb{O}$ requires a specific  approach, where the \emph{real part} of operators and functions plays a central role.} We establish left and right functional calculi $\Phi$ and $\Psi$, prove {some} core properties, and demonstrate independence from spectral decompositions.

\subsection{Preliminaries and Function Space}
\label{subsec:prelim}

Let $H$ be a Hilbert $\mathbb{O}$-bimodule and $T \in \mathscr{B}_{\mathcal{RO}}(H)$ a compact self-adjoint para-linear operator whose eigenvalues are \emph{standard strong eigenvalues} (Definition \ref{def:ssEngen}). By Theorem \ref{thm:hilbert-schmidt}, there is a weak associative orthonormal basis consisting of strong eigenvectors $\{z_i\}_{i \in \mathcal{N}} \cup \{z_{\alpha}\}_{\alpha \in \Lambda}$ such that
$$Tz_i=\lambda_iz_i, i\in \mathcal{N} \qquad Tz_\alpha=0, \alpha\in \Lambda,$$ and   $T$ admits a spectral decomposition:
\[
T = \sum_{i=1}^{\infty} \lambda_i P_{z_i}.
\]
The \emph{spectrum} $\sigma(T) = \{\lambda_i\} \cup \{0\}$ is a compact subset of $\mathbb{R}$.

Define $\mathcal{R}$ as the algebra of $\mathbb{O}$-valued  functions $f$ on  $\sigma(T)$ such that $$\sup_{q \in \sigma(T)} |f(q)|< +\infty.$$
 Equip $\mathcal{R}$ with:
\begin{itemize}
    \item \textbf{Octonionic multiplication}:
    \[
    (p  f)(q) := pf(q), \quad (f  p)(q) := f(q)p, \quad p \in \mathbb{O}.
    \]
    \item \textbf{Product}: For $g\in \mathcal{R}$,
    \[
    (f  g)(q) := f(q)g(q).
    \]
    \item \textbf{Involution}: $$f^*(q) := \overline{f(q)}.$$
    \item \textbf{Norm}: $$\|f\| := \sup_{q \in \sigma(T)} |f(q)|.$$
\end{itemize}
Then $\mathcal{R}$ is an $\mathbb{O}$-involutive Banach algebra. The \emph{real part} $\operatorname{Re} \mathcal{R}$ consists of real-valued functions.

\subsection{Right and Left Functional Calculi}
\label{subsec:def_calculi}

The non-associativity of $\mathbb{O}$ necessitates distinct left and right functional calculi, defined via the spectral decomposition. Note that   $\{z_i\}_{i \in \mathcal{N}} \cup \{z_{\alpha}\}_{\alpha \in \Lambda}$ is a weak associative orthonormal basis. This enables us to define   $$P_0:=\sum_{\alpha\in\Lambda}P_{z_\alpha}.$$

\begin{mydef}[Functional Calculi]
\label{def:functional_calculi}
The \emph{right functional calculus} $\Phi: \mathcal{R} \to \mathscr{B}_{\mathcal{RO}}(H)$ is
{defined by}
\[
\Phi(f)(x):= (P_0\odot f(0))(x)+\sum_{i=1}^{\infty} (P_{z_i} \odot f(\lambda_i))(x),
\]
where $\odot$ is the octonionic scalar multiplication for operators (see Theorem \ref{thm:para_bimodule}). The series converge uniformly in  norm by Parseval's Theorem (Theorem  \ref{thm:parseval}).
For simplicity, we can write $$\Phi(f) = \sum_{i\in \mathcal{N}\cup \Lambda } P_{z_i} \odot f(\lambda_i).$$

The \emph{left functional calculus} $\Psi: \mathcal{R} \to \mathscr{B}_{\mathcal{RO}}(H)$ is
{defined by}
\[
\Psi(f)(x) := (f(0)\odot P_0)(x)+\sum_{i=1}^{\infty} (f(\lambda_i) \odot P_{z_i})(x).
\]
\end{mydef}
\subsection{Key Properties}
\label{subsec:key_properties}

The functional calculi satisfy the following properties, generalizing the spectral theorem to octonionic spaces.

\begin{thm}[Properties of  the functional calculi]
\label{thm:functional_calculus_properties}
The map $\Phi$ satisfies:
\begin{enumerate}
    \item \textbf{Power Preservation}: For all $k \in \mathbb{N}$,
    \[
    \Phi(q^k) = T^{\circledcirc k},
    \]

   {where $T^{\circledcirc k}$ is the unique para-linear operator satisfying $T^{\circledcirc  k}(x) = T^k(x)$ for $x \in \operatorname{Re} H$.}
    \item \textbf{Right Para-Linearity}: For $f, g \in \mathcal{R}$, $p \in \mathbb{O}$,
    \[
    \Phi(f + g) = \Phi(f) + \Phi(g), \quad \operatorname{Re} \left( \Phi(f) \odot p - \Phi(f  p) \right) = 0.
    \]

    \item \textbf{{Real-part Multiplicative Property}}: For $f \in \operatorname{Re} \mathcal{R}$, $g \in \mathcal{R}$,
    \[
    \operatorname{Re} \Phi(f  g) = \operatorname{Re} \left( \Phi(f) \circledcirc \Phi(g) \right).
    \]
Note that here $   \operatorname{Re}$ represents the real part operator on $ \mathscr{B}_{\mathcal{RO}}(H)$.
    \item \textbf{Involution}: $$\Phi(f^*) = \Psi(f)^*.$$

    \item \textbf{Norm Bounds}:
  \begin{eqnarray*}
    &\|\Phi(f)\| \leqslant \|f\| \qquad  &\text{ for }\quad  f \in \operatorname{Re} \mathcal{R};
    \\
    & \|\Phi(f)\| \leqslant  8 \|f\|\qquad  &\text{ for }\quad  f \in \mathcal{R}.
  \end{eqnarray*}
\end{enumerate}
Analogous results hold for $\Psi$.
\end{thm}

\begin{proof}\par\noindent
 \text{(1)} Proceed by induction. For $k=0,1$, $\Phi(1) = \text{Id}$ by Theorem  \ref{thm:parseval}, and $\Phi(q) = T$ by  Theorem \ref{thm:hilbert-schmidt}. Assume $\Phi(q^k) = T^{\circledcirc  k}$ for some $k\geqslant 1$. For $x \in \operatorname{Re} H$, by the definition of $T^{\circledcirc k}$ and the induction hypothesis,
\[
T^{\circledcirc (k+1)}(x) = T(T^k(x)) =T(T^{\circledcirc k}(x))= T\left( \Phi(q^k)(x) \right) = T\left( \sum_{i=1}^\infty \lambda_i^k P_{z_i}(x) \right).
\]
In view of   \eqref{eq:projection_orthogonality}, we have
\[
T\left( \sum_i \lambda_i^k P_{z_i}(x) \right) = \sum_i \lambda_i^k T(P_{z_i}(x)) = \sum_i \lambda_i^k (T \circledcirc P_{z_i})(x) = \sum_i \lambda_i^{k+1} P_{z_i}(x) = \Phi(q^{k+1})(x),
\]
where we used Theorem \ref{thm:projection-commutation} {to deduce} $T \circledcirc  P_{z_i} = \lambda_i P_{z_i}$. Since both sides are para-linear and agree on $\operatorname{Re} H$, they are equal by Lemma \ref{cor: A_p(x,f)=0 and f(px)=pf(x)}(2).

\text{(2)} Additivity is immediate from linearity of the series. For scalar homogeneity,
\begin{eqnarray*}
	\Phi(f  p) &=& P_0\odot(f(0)p)+\sum_i P_{z_i} \odot \left( f(\lambda_i) p \right) \\
	&=& (P_0\odot f(0))\odot p-\left[ P_{0}, f(0), p \right]_{\mathscr{B}_{\mathcal{RO}}(H)}+\sum_i \left( P_{z_i} \odot f(\lambda_i) \right) \odot p - \left[ P_{z_i}, f(\lambda_i), p \right]_{\mathscr{B}_{\mathcal{RO}}(H)}\\
	&=&\Phi(f)\odot p-\left[ P_{0}, f(0), p \right]_{\mathscr{B}_{\mathcal{RO}}(H)}-\sum_i  \left[ P_{z_i}, f(\lambda_i), p \right]_{\mathscr{B}_{\mathcal{RO}}(H)}.
\end{eqnarray*}
 Thus by the continuity of the real part operator on ${\mathscr{B}_{\mathcal{RO}}(H)}$,
\[
\operatorname{Re} \left( \Phi(f) \odot p - \Phi(f  p) \right) =   \sum_i\operatorname{Re} \left[ P_{z_i}, f(\lambda_i), p \right]_{\mathscr{B}_{\mathcal{RO}}(H)}  = 0.
\]
	Here we use the fact that ${\mathscr{B}_{\mathcal{RO}}(H)}$ is an $\O$-bimodule (Theorem \ref{thm:para_bimodule}) which yields the real part of the associators vanish.

 \textbf{(3)} Let \(f\in \operatorname{Re}\mathcal{R}\) and \(g\in \mathcal{R}\). Then \(f(\lambda_{i})\in \mathbb{R}\) for all \(i\). Assume first that \(f(0)=g(0)=0\). Using Lemma~\ref{lem:Rp prop 2} and the associator identity \eqref{eq:re [fgh]=0}, we obtain
\begin{align*}
\operatorname{Re}\!\big(\Phi(f)\circledcirc \Phi(g)\big)
&= \operatorname{Re}\!\left(\sum_{i,j} f(\lambda_{i}) \,\Big(P_{z_{i}} \circledcirc \big(P_{z_{j}}\odot g(\lambda_{j})\big)\Big)\right) \\
&= \sum_{i,j} f(\lambda_{i})\, \operatorname{Re}\!\left(P_{z_{i}}\circledcirc \big(P_{z_{j}}\circledcirc L_{g(\lambda_{j})}\big)\right) \\
&= \sum_{i,j} f(\lambda_{i})\, \operatorname{Re}\!\left(\big(P_{z_{i}}\circledcirc P_{z_{j}}\big)\circledcirc L_{g(\lambda_{j})}\right).
\end{align*}
Applying \eqref{eq:projection_orthogonality} and Lemma~\ref{lem:Rp prop 2} yields
\[
\operatorname{Re}\!\big(\Phi(f)\circledcirc \Phi(g)\big)
= \sum_{i} f(\lambda_{i})\, \operatorname{Re}\!\left(P_{z_{i}}\circledcirc L_{g(\lambda_{i})}\right)
= \operatorname{Re}\!\left(\sum_{i} P_{z_{i}}\odot \big(f(\lambda_{i})g(\lambda_{i})\big)\right)
= \operatorname{Re}\,\Phi(fg).
\]
{{For general \(f,g\), write \(f=a+f_{1}\) and \(g=b+g_{1}\) with \(a\in \mathbb{R}\), \(b\in \mathbb{O}\), and \(f_{1}(0)=g_{1}(0)=0\). Using the right para-linearity of \(\Phi\) (Theorem~\ref{thm:functional_calculus_properties}\,(2)) and the equality \(\Phi(p)=\mathrm{Id}\odot p\) for \(p\in \mathbb{O}\), we reduce to the previous case and conclude
\[
\operatorname{Re}\,\Phi(fg)=\operatorname{Re}\!\big(\Phi(f)\circledcirc \Phi(g)\big).
\]
}}

\text{(4)} By Theorem \ref{thm:involutive_algebra}  and $\lambda_i \in \mathbb{R}$, {we deduce}
\[
\Phi(f^*) = P_0\odot \overline{f(0)}+\sum_i P_{z_i} \odot \overline{f(\lambda_i)} =(f(0)\odot P_0)^*+ \sum_i \left( f(\lambda_i) \odot P_{z_i} \right)^* = \Psi(f)^*.
\]

\text{(5)} For $f \in \operatorname{Re} \mathcal{R}$, Parseval's Theorem (Theorem  \ref{thm:parseval}) gives
\begin{eqnarray*}
	\|\Phi(f)(x)\|^2 &=& \left\|\sum_{\alpha\in \Lambda} z_{\alpha} \langle z_{\alpha}, x \rangle f(0) + \sum_{i\in \mathcal{N}} z_i \langle z_i, x \rangle f(\lambda_i) \right\|^2 \\
	&=& \sum_{\alpha\in \Lambda} |\langle z_{\alpha}, x \rangle|^2 |f(0)|^2 +\sum_{i\in \mathcal{N}} |\langle z_i, x \rangle|^2 |f(\lambda_i)|^2 \\
	&\leqslant& \|f\|^2 \|x\|^2.
	\end{eqnarray*}
Thus $\|\Phi(f)\| \leqslant  \|f\|$.

For general $f$, decompose $f = \sum_{k=0}^7 f_{(k)}  e_k$ with $f_{(k)} \in \operatorname{Re} \mathcal{R}$. By Theorem \ref{thm:B-O-Module},
\[
\|\Phi(f)\| \leqslant \sum_{k=0}^7 \|\Phi(f_{(k)}) \odot e_k\| \leqslant \sum_{k=0}^7 |e_k| \,  \|\Phi(f_{(k)})\| \leqslant  8 \|f\|.
\]
\end{proof}

\subsection{Independence of Spectral Decomposition}
\label{subsec:independence}

The functional calculus is independent of the choice of spectral decomposition.

\begin{thm}
\label{thm:independence}
$\Phi(f)$ is independent of the orthonormal basis $\{z_i\}$ in the spectral decomposition of $T$.
\end{thm}

\begin{proof}
For $f \in \operatorname{Re} \mathcal{R}$, approximate uniformly by polynomials $f_n(q) = \sum_{k=0}^n q^k a_k$ with $a_k \in \mathbb{R}$. By Theorem \ref{thm:functional_calculus_properties}(1), $\Phi(f_n) = \sum_{k=0}^n a_k T^{\circledcirc k}$ depends only on $T$, not the basis. Since $\Phi$ is continuous ($\|\Phi(f) - \Phi(g)\| \leqslant  8 \|f - g\|$),
\[
\Phi(f) = \lim_{n \to \infty} \Phi(f_n)
\]
is independent of the basis.

 For general $f\in  \mathcal{R}$, decompose $f = \sum_{k=0}^7 f_{(k)}  e_k$ with $f_{(k)} \in \operatorname{Re} \mathcal{R}$. By right para-linearity,
\[
\Phi(f) = \sum_{k=0}^7 \Phi(f_{(k)}) \odot e_k,
\]
and each $\Phi(f_{(k)})$ is independent of the basis.
\end{proof}

\subsection{Remark}
\label{subsec:discussion}

The functional calculi $\Phi$ and $\Psi$ exhibit intrinsic non-associative features, in particular:
\begin{itemize}
    \item \textbf{Homomorphism limitation}: Multiplicativity holds only for the real part (Theorem \ref{thm:functional_calculus_properties}(3));
    \item \textbf{Left-right distinction}: $\Phi \neq \Psi$ but $\Phi(f^*) = \Psi(f)^*$ (Theorem \ref{thm:functional_calculus_properties}(4));
    \item \textbf{Norm deviation}: The factor $8$ in $\|\Phi(f)\|$ quantifies octonionic non-associativity.
\end{itemize}
This framework provides a consistent spectral-based functional calculus, laying foundations for extensions to broader operator classes in non-associative functional analysis.

\section{Concluding Remarks}
\label{sec:conclusion}

{In this section we summarize the main contributions of this paper.} This work resolves fundamental challenges in non-associative functional analysis by establishing the first complete spectral theory for self-adjoint operators on octonionic Hilbert spaces, a problem {which has been} open since Goldstine and Horwitz's foundational 1964 work. The core difficulty stemmed from the incompatibility of associative operator theory with non-associativity of octonions, which obstructed spectral decomposition, adjoint operators, and {a consistent definition of a} functional calculus. Our {approach is based} on \textit{para-linearity}, a geometrically motivated algebraic framework that respects octonionic structure while enabling robust operator theory. The principal contributions are:

\subsection*{Key Innovations}
\begin{enumerate}[leftmargin=*,align=left]
\item[\bfseries (1)] \textit{Para-linear formalism.} We introduce para-linearity (Definition       \ref{def:para_linear}) as the definitive non-associative analog of linearity, characterized by the vanishing real part of the second associator:
\[
\operatorname{Re}(B_p(T,x)) = 0, \quad B_p(T,x) := T(x)p - T(xp).
\]
This recovers the essential analytic {results, among which} Riesz representation (Theorem \ref{thm:riesz}), Parseval's theorem (Theorem \ref{thm:parseval}), and Hahn-Banach extension  Theorem (see \cite{huo2025BLMSHB}).

\item[\bfseries (2)] \textit{Involutive operator algebra.} We construct the first octonionic involutive Banach algebra $\mathscr{B}_{\mathcal{RO}}(H)$ of bounded para-linear operators (Theorems \ref{thm:norm_estimate} and 	\ref{thm:involutive_algebra}) via {the following tools:}
\begin{itemize}
\item \textit{Regular composition} $\circledcirc$ (Definition 	\ref{def:regular_composition}), defined by $$\operatorname{Re}((f \circledcirc g)(x)) = \operatorname{Re}(f(g(x))),$$ which resolves associativity failures.
\item A corrected adjoint $T^*$ (Definition 	\ref{def:adjoint_operator}) satisfying $$\langle x, T^*y \rangle = \langle Tx, y \rangle + [x,y,T],$$ where the associator $[x,y,T]$ quantifies non-associative deviation.
\end{itemize}
The algebra has norm $\|T\|' = 8\|T\|$ and scalar multiplications $p \odot T$, $T \odot p$ compatible with octonionic bimodule structure.

\item[\bfseries (3)] \textit{Geometric spectral theory.} We prove an \textit{octonionic polarization identity} (Theorem 	\ref{thm:polarization}):
\[
\langle Tx, y \rangle = \frac{1}{56} \sum_{\substack{i,j=1 \\ i \neq j}}^7 (2A_{ij} + B_{ij} + C_{ij})(x,y) - \frac{1}{98} \sum_{i,j=1}^7 A_{ij}(x,y) - \frac{1}{2} m(x,y)
\]
for $x,y \in \operatorname{Re} H$, demonstrating that sesquilinear forms are determined by quadratic forms restricted to the \textit{slice cone}:
\[
\mathbb{C}(H) = \bigcup_{J \in \mathbb{S}} (\operatorname{Re} H + J \operatorname{Re} H), \quad \mathbb{S} = \{J \in \operatorname{Im} \mathbb{O} : |J| = 1\}.
\]
This yields a geometric characterization of self-adjointness (Theorem \ref{thm:slice_cone_char}): $T = T^*$ if and only if $\langle Tx, x \rangle \in \mathbb{R}$ for all $x \in \mathbb{C}(H)$.

\item[\bfseries (4)] \textit{Spectral decomposition.} For compact, self-adjoint, para-linear operators with \textit{standard strong eigenvalues} (Definition \ref{def:ssEngen}), we establish a spectral theorem (Theorem \ref{thm:hilbert-schmidt}):
\[
T = \sum_{i \in \mathcal{N}} \lambda_i P_{z_i}, \quad P_{z_i}(x) = z_i \langle z_i, x \rangle,
\]
where $\lambda_i \in \mathbb{R}$ are eigenvalues, $\{z_i\} \subset \mathbb{C}(H)$ are orthonormal strong eigenvectors, and convergence is norm-topological. Eigenvectors exhibit \textit{weak associativity}: $B_p(z_i, z_j) = 0$ for all $p \in \mathbb{O}$ (Theorem \ref{thm:eigen-orthogonality}).

\item[\bfseries (5)] \textit{Functional calculus.} We develop left and right functional calculi $\Psi$, $\Phi$ (Theorem \ref{thm:functional_calculus_properties}) for octonion-valued functions $f \in \mathcal{R}$:
\[
\Phi(f) = \sum_{i\in \mathcal{N}\cup \Lambda } P_{z_i} \odot f(\lambda_i), \quad \Psi(f) = \sum_{i\in \mathcal{N}\cup \Lambda } f(\lambda_i) \odot P_{z_i},
\]
satisfying
\begin{itemize}
\item Power preservation: $\Phi(q^k) = T^{\circledcirc k}$.
\item Para-linearity: $\re \Phi(f p)-\Phi(f)\odot p=0$.
\item Real-multiplicativity: $\operatorname{Re} \Phi(f  g) = \operatorname{Re} (\Phi(f) \circledcirc \Phi(g))$ for $f \in \operatorname{Re} \mathcal{R}$.
\item Norm bound: $\|\Phi(f)\| \leqslant  8 \|f\|_\infty$.
\end{itemize}
The calculi are decomposition-independent (Theorem \ref{thm:independence}).
\end{enumerate}

In the forthcoming article \cite{hrs2025Omoore}, we will apply our para-linear operator framework to the spectral theory of the Albert algebra, the 27-dimensional exceptional simple Jordan algebra,  opening new avenues in exceptional geometry, representation theory, and Jordan-based approaches to quantum physics.



\begin{thebibliography}{10}

\bibitem{MR1333599}
S.~L. Adler.
\newblock {\em Quaternionic quantum mechanics and quantum fields}, volume~88 of
  {\em International Series of Monographs on Physics}.
\newblock The Clarendon Press, Oxford University Press, New York, 1995.

\bibitem{MR4771382}
S.~Alesker and P.~V. Gordon.
\newblock Octonionic {C}alabi-{Y}au theorem.
\newblock {\em J. Geom. Anal.}, 34(9):Paper No. 293, 65, 2024.

\bibitem{Alpay2016JMP}
D.~Alpay, F.~Colombo, and D.~P. Kimsey.
\newblock The spectral theorem for quaternionic unbounded normal operators
  based on the {$S$}-spectrum.
\newblock {\em J. Math. Phys.}, 57(2):023503, 27, 2016.

\bibitem{baez2002octonions}
J.~C. Baez.
\newblock The octonions.
\newblock {\em Bull. Amer. Math. Soc. (N.S.)}, 39(2):145--205, 2002.

\bibitem{Baez2012FP}
J.~C. Baez.
\newblock Division algebras and quantum theory.
\newblock {\em Found. Phys.}, 42(7):819--855, 2012.

\bibitem{bryant2003some}
R.~L. Bryant.
\newblock Some remarks on {$G_2$}-structures.
\newblock In {\em Proceedings of {G}\"{o}kova {G}eometry-{T}opology
  {C}onference 2005}, pages 75--109. G\"{o}kova Geometry/Topology Conference
  (GGT), G\"{o}kova, 2006.

\bibitem{Colombo2019normal}
P.~Cerejeiras, F.~Colombo, U.~K\"{a}hler, and I.~Sabadini.
\newblock Perturbation of normal quaternionic operators.
\newblock {\em Trans. Amer. Math. Soc.}, 372(5):3257--3281, 2019.

\bibitem{MR4855317}
B.~Collier and J.~Toulisse.
\newblock Holomorphic curves in the 6-pseudosphere and cyclic surfaces.
\newblock {\em Trans. Amer. Math. Soc.}, 377(9):6465--6514, 2024.

\bibitem{MR3967697}
F.~Colombo and J.~Gantner.
\newblock {\em Quaternionic closed operators, fractional powers and fractional
  diffusion processes}, volume 274 of {\em Operator Theory: Advances and
  Applications}.
\newblock Birkh\"{a}user/Springer, Cham, 2019.

\bibitem{MR3887616}
F.~Colombo, J.~Gantner, and D.~P. Kimsey.
\newblock {\em Spectral theory on the {S}-spectrum for quaternionic operators},
  volume 270 of {\em Operator Theory: Advances and Applications}.
\newblock Birkh\"auser/Springer, Cham, 2018.

\bibitem{Colombo2023OctonionicMA}
F.~Colombo, R.~S. Krau{\ss}har, and I.~Sabadini.
\newblock Octonionic monogenic and slice monogenic {H}ardy and {B}ergman
  spaces.
\newblock {\em Forum Math.}, 36(4):1031--1052, 2024.

\bibitem{Colombo2011ADVqevol}
F.~Colombo and I.~Sabadini.
\newblock The quaternionic evolution operator.
\newblock {\em Adv. Math.}, 227(5):1772--1805, 2011.

\bibitem{Colombo2008funcalculus}
F.~Colombo, I.~Sabadini, and D.~C. Struppa.
\newblock A new functional calculus for noncommuting operators.
\newblock {\em J. Funct. Anal.}, 254(8):2255--2274, 2008.

\bibitem{colombo2011noncomfunctcalculus}
F.~Colombo, I.~Sabadini, and D.~C. Struppa.
\newblock {\em {N}oncommutative functional calculus}, volume 289 of {\em
  Progress in Mathematics}.
\newblock Birkh\"{a}user/Springer Basel AG, Basel, 2011.
\newblock Theory and applications of slice hyperholomorphic functions.

\bibitem{Dixon1994division}
G.~M. Dixon.
\newblock {\em {D}ivision algebras: octonions, quaternions, complex numbers and
  the algebraic design of physics}, volume 290 of {\em Mathematics and its
  Applications}.
\newblock Kluwer Academic Publishers Group, Dordrecht, 1994.

\bibitem{Duff1999world}
M.~J. Duff, editor.
\newblock {\em The world in eleven dimensions: supergravity, supermembranes and
  {M}-theory}.
\newblock Studies in High Energy Physics Cosmology and Gravitation. IOP
  Publishing, Bristol, 1999.

\bibitem{Furey2022standardmod}
N.~Furey and M.~J. Hughes.
\newblock One generation of standard model {W}eyl representations as a single
  copy of ${R}\otimes{C}\otimes{H}\otimes{O}$.
\newblock {\em Phys. Lett. B}, 827:Paper No. 136959, 2022.

\bibitem{MR3587903}
R.~Ghiloni, V.~Moretti, and A.~Perotti.
\newblock Spectral properties of compact normal quaternionic operators.
\newblock In {\em Hypercomplex analysis: new perspectives and applications},
  Trends Math., pages 133--143. Birkh\"{a}user/Springer, Cham, 2014.

\bibitem{MR2737796}
R.~Ghiloni and A.~Perotti.
\newblock Slice regular functions on real alternative algebras.
\newblock {\em Adv. Math.}, 226(2):1662--1691, 2011.

\bibitem{goldstine1964hilbert}
H.~H. Goldstine and L.~P. Horwitz.
\newblock Hilbert space with non-associative scalars. {I}.
\newblock {\em Math. Ann.}, 154:1--27, 1964.

\bibitem{gunaydin2013FP}
M.~G\"unaydin and D.~Minic.
\newblock Nonassociativity, {M}alcev algebras and string theory.
\newblock {\em Fortschr. Phys.}, 61(10):873--892, 2013.

\bibitem{huoqinghai2022Riesz}
Q.~Huo and G.~Ren.
\newblock Para-linearity as the nonassociative counterpart of linearity.
\newblock {\em J. Geom. Anal.}, 32(12):Paper No. 304, 30, 2022.

\bibitem{huoqinghai2021tensor}
Q.~Huo and G.~Ren.
\newblock Structure of octonionic {H}ilbert spaces with applications in the
  {P}arseval equality and {C}ayley--{D}ickson algebras.
\newblock {\em J. Math. Phys.}, 63(4):Paper No. 042101, 24, 2022.

\bibitem{huo2024aacasubmod}
Q.~Huo and G.~Ren.
\newblock On octonionic submodules generated by one element.
\newblock {\em Adv. Appl. Clifford Algebr.}, 34(5):Paper No. 46, 12, 2024.

\bibitem{huoqinghai2020nonass}
Q.~Huo and G.~Ren.
\newblock Non-associative {C}ategories of {O}ctonionic {B}imodules.
\newblock {\em Commun. Math. Stat.}, 13(2):303--369, 2025.

\bibitem{huo2025BLMSHB}
Q.~Huo and G.~Ren.
\newblock Octonionic {H}ahn-{B}anach theorem for para-linear functionals.
\newblock {\em Bull. Lond. Math. Soc.}, 57(2):472--489, 2025.

\bibitem{hrs2025Omoore}
Q.~Huo, G.~Ren, and I.~Sabadini.
\newblock Algebra meets analysis in the spectral theory of the {A}lbert
  algebra.
\newblock {\em preprint}.

\bibitem{MR4935003}
Q.~Huo, G.~Ren, and I.~Sabadini.
\newblock Octonionic {H}ilbert spaces and para-linear operators.
\newblock In {\em Hypercomplex analysis and its applications}, Trends Math.,
  pages 93--103. Birkh\"auser/Springer, Cham, [2025].

\bibitem{MR4782804}
D.~Jang.
\newblock Circle actions on oriented manifolds with 3 fixed points.
\newblock {\em Int. Math. Res. Not. IMRN}, (15):11371--11385, 2024.

\bibitem{Krausshar2022discoct}
R.~S. Krau{\ss}har, A.~Legatiuk, and D.~Legatiuk.
\newblock Towards discrete octonionic analysis.
\newblock In {\em Differential equations, mathematical modeling and
  computational algorithms}, volume 423 of {\em Springer Proc. Math. Stat.},
  pages 51--63. Springer, Cham, 2023.

\bibitem{Krauhar2023WeylCP}
R.~S. Krau{\ss}har and D.~Legatiuk.
\newblock Weyl calculus perspective on the discrete stokes' formula in
  octonions.
\newblock In {\em Computer Graphics International Conference}, 2023.

\bibitem{Moretti2019quantum}
V.~Moretti.
\newblock {\em Fundamental mathematical structures of quantum theory}.
\newblock Springer, Cham, 2019.
\newblock Spectral theory, foundational issues, symmetries, algebraic
  formulation.

\bibitem{Springer2000jordanalg}
T.~A. Springer and F.~D. Veldkamp.
\newblock {\em Octonions, {J}ordan algebras and exceptional groups}.
\newblock Springer Monographs in Mathematics. Springer-Verlag, Berlin, 2000.

\bibitem{semrl1986Hquadratic}
P.~\v{S}emrl.
\newblock On quadratic and sesquilinear functionals.
\newblock {\em Proc. Amer. Math. Soc.}, 31(1):184--190, 1986.

\end{thebibliography}
				

			\end{document}